\documentclass[10pt]{amsart}
\usepackage[a4paper, total={6in, 8in}]{geometry}
\usepackage{amsmath}
\usepackage{amssymb}
\usepackage{amsfonts}
\usepackage{mathrsfs}
\usepackage{amscd}
\usepackage{graphicx}
\usepackage[shortlabels]{enumitem}
\usepackage{mathtools}
\usepackage{tikz-cd}
\usepackage{cleveref}
\usepackage{pstricks-add}
\usepackage{pgf,tikz}
\usepackage{subcaption}
\usepackage{bm}

\usepackage{tkz-tab}
\usepackage{xpatch}
\xpatchcmd{\tkzTabLine}{$0$}{$\bullet$}{}{}
\tikzset{t style/.style={style=solid}}

\usetikzlibrary{arrows}

\crefformat{section}{\S#2#1#3} 
\crefformat{subsection}{\S#2#1#3}
\crefformat{subsubsection}{\S#2#1#3}

\numberwithin{equation}{section}       

\theoremstyle{plain} 
\newcommand{\thistheoremname}{}
\newtheorem*{genericthm*}{\thistheoremname}
\newenvironment{namedtheorem*}[1]
  {\renewcommand{\thistheoremname}{#1}%
   \begin{genericthm*}}
  {\end{genericthm*}}

\theoremstyle{plain}
\newtheorem{theo}{Theorem}
\newtheorem{prop}{Proposition}[section]

\newtheorem{lemm}[prop]{Lemma}

\theoremstyle{definition}

\newtheorem{defi}[prop]{Definition}
\theoremstyle{remark}
\newtheorem{rema}[prop]{Remark}
\newtheorem{exam}[prop]{Example}
\newtheoremstyle{citing}
  {3pt}
  {3pt}
  {\itshape}
  {}
  {\bfseries}
  {.}
  {.5em}
  {\thmnote{#3}}
\theoremstyle{citing}
\newcommand{\C}{\mathbb{C}}
\newcommand{\D}{\mathbb{D}}
\newcommand{\F}{\mathbb{F}}

\newcommand{\N}{\mathbb{N}}

\newcommand{\Q}{\mathbb{Q}}
\newcommand{\R}{\mathbb{R}}

\newcommand{\Z}{\mathbb{Z}}
\newcommand{\cA}{\mathcal{A}}
\newcommand{\cB}{\mathcal{B}}

\newcommand{\cD}{\mathcal{D}}

\newcommand{\cF}{\mathcal{F}}

\newcommand{\cI}{\mathcal{I}}

\newcommand{\cK}{\mathcal{K}}
\newcommand{\cL}{\mathcal{L}}

\newcommand{\cO}{\mathcal{O}}
\newcommand{\cP}{\mathcal{P}}

\newcommand{\cR}{\mathcal{R}}

\newcommand{\cU}{\mathcal{U}}

\newcommand{\teta}{\widetilde{\teta}}

\newcommand{\e}{\varepsilon}

\newcommand{\ov}{\overline}
\newcommand{\ovra}{\overrightarrow}
\renewcommand{\=}{ : = }
\renewcommand{\bf}{\mathbf}

\DeclareMathOperator{\diam}{diam}

\DeclareMathOperator{\dist}{dist}

\DeclareMathOperator{\interior}{int}
\DeclareMathOperator{\Crit}{Crit} 

\DeclareMathOperator{\sg}{sgn}
\DeclareMathOperator{\Trn}{Trn}
\DeclareMathOperator{\diff}{Diff}
\DeclareMathOperator{\Trace}{Tr}

\newcommand{\CC}{\overline{\C}}

\newcommand{\SFib}{\Sigma_{\text{Fib}}}
\tikzset{
  declare function={
    sgn(\x) = (and(\x<0, 1) * -1) +
    (and(\x>0, 1) * 1) +
    (and(\x==0, 1) * 0);
  }
}

\usepackage{lipsum}
\begin{document}


\usetikzlibrary{shapes, arrows, calc, arrows.meta, fit, positioning, quotes} 
\tikzset{  
    state/.style ={ellipse, draw, minimum width = 0.9 cm}, 
    point/.style = {circle, draw, inner sep=0.18cm, fill, node contents={}},  
    bidirected/.style={Latex-Latex,dashed}, 
    el/.style = {inner sep=2.5pt, align=right, sloped}  
}  

\colorlet{ColorGray}{gray!10}

\title{The Artin-Mazur zeta function for interval maps}
\author{Jorge Olivares-Vinales} \thanks{} 
\address{Shanghai Center for Mathematical Sciences, Jiangwan Campus, Fudan University, No 2005 Songhu Road, Shanghai, China 200438 }
\email{jolivaresv@fudan.edu.cn}

\begin{abstract}
  In this work we study the Artin-Mazur zeta function for piecewise monotone
functions acting on a compact interval of real numbers. In the case of 
unimodal maps, Milnor and Thurston \cite{MiTh88} gave a characterization
for the rationality of the Artin-Mazur zeta function in terms of the 
orbit of the unique turning point.  We show that for multimodal maps, the 
previous characterization does not hold.
\end{abstract}

\maketitle
\section{Introduction}
\label{Section_Introduction}

Given a continuous self map $f$ from a compact topological space $X$, a 
classical problem in dynamical systems is to count the number of periodic
orbits $\{ x, f(x), \ldots , f^{n-1}(x) \}$ with $x \in X$, $f^n(x) = x$,
and $n \in \Z_+$. An important approach to this problem was proposed by 
Artin and Mazur \cite{Artin-Mazur1965} based on an analogy with the Weil 
zeta function of an algebraic variety over a finite field (see 
\cite{Weil1949}). 
They defined a function of a single complex variable in the following way:  
Suppose that for each integer $n \geq 1 $ we have that $f^n$ has only 
finitely many fixed points. Then, the \emph{Artin-Mazur zeta function} of
$f$ is the formal power series 
\[ \zeta_f(t) \= \exp \left( \sum_{n \geq 1} \frac{1}{n}N_n(f)t^n \right), \]
where $N_n(f)$ is the number of isolated fixed points of $f^n$. 
If $X$ is a smooth compact manifold (without boundary), then Artin and 
Mazur showed that for a dense subset of 
$f \in \diff^{\infty}(X)$ the power series $\zeta_f(t)$ has a positive 
radius of convergence. In this context, Smale 
\cite[Problem 4.5]{Smale1967_diff_dyn_systems} raised the
question if the Artin-Mazur zeta fuction is \emph{generically} rational. 
Meyer \cite{Meyer1967_Periodic_points_of_diffeomorphisms}, 
Bowen and Lanford \cite{Bowen_Lanford_Zeta_function_of_subshift}, and 
Bowen \cite{Bowen1970_Topological_entropy_Axiom_A},  proved that for 
Axiom A diffeomorphisms the A-M zeta function has a positive radius of 
convergence and (restricted to some basic set) it is rational. Williams 
\cite{Williams168_The_zeta_function_of_an_attractor}
proved that the A-M zeta function for Anosov diffeomorphisms is 
rational. Finally,
Guckenheimer \cite{Guckenheimer1970_Zeta_function} and 
Manning \cite{Manning1971_Axiom_A_diffeos_have_rational_zeta_functions}
gave a positive answer to Smale Conjecture. See also 
\cite{Franks1978_A_reduce_zeta_function_for_diffeomorphisms} for an 
alternative proof using a reduced zeta function and 
\cite{Fried_Rationality_for_isolated_expansive_sets} for a proof that 
makes no use of the symbolic dynamics arising from Markov partitions, and 
thus it applies in a more general context.

Recently Berger 
\cite{Berger2021_Generic_family_displayin_robustly_a_fast_growth_of_per_point} 
proved that, in a measure-theoretic sense (Arnold's typicality), almost every 
diffeomorphism has a divergent Artin–Mazur zeta function. See also,
\cite{Kaloshin-Yu2000} for the proof for prevalent families, 
\cite{Kaloshin_Hunt2007_Stretched_exponential_estimates_on_growth_of_the_number_of_per} for the residual case.



In the context of piecewise monotone maps of the interval, Milnor and 
Thurston compute the Artin-Mazur zeta function in terms of the 
combinatorial information of the forward orbit of the turning points of
the map encoded in a power series called the 
\emph{kneading determinant} (see Section \ref{subsec_multimodal_Kneading_thoery}). 
When $f$ is a unimodal map, the 
kneading determinant depends on the forward orbit of the unique turning 
point (the kneading invariant of the map).
Using this, they made the following statement 
(see \cite[Section 8]{MiTh88}) 

\vspace{0.5cm}

\emph{"If $f\colon I \longrightarrow I$ is a
$C^3-$unimodal map with negative Schwarzian derivative such that 
$f(\partial I) \subset \partial I$ and the fixed boundary point is 
unstable. Then there exists an attracting or one-sided attracting 
periodic orbit if and only if the kneading determinant of $f$ is 
periodic."} 

\vspace{0.5cm}

In particular, the above statement implies the following characterization
for quadratic polynomials acting on $\R$ (See Theorem \ref{theo:Milnor_Thurston_kneading_determinant_zeta_function_theorem} and 
\cite[Theorem 9.1]{MiTh88})

\begin{theo}
    For the quadratic map $P_c(x) = x^2 + c$, with $c \in \R$, the 
    following are equivalent:
    \begin{enumerate}
        \item[(1)] The Artin-Mazur zeta function $\zeta_{P_c}(t)$ is 
        rational;
        \item[(2)] the map $P_c$ is hyperbolic or post-critically finite.
    \end{enumerate}
\end{theo}

Flatto and Lagarias \cite{Flatto_Lagarias2000_The_lap-counting_and_zeta_function_of_the_tent_map} studied the Artin-Mazur zeta
function
for the tent family (symmetric piecewise linear unimodal maps) using the 
lap-counting 
function (a generating function for the lap number of the map) associated
to a linear mod $1$ transformation. In particular they obtained that the 
Artin-Mazur zeta function of a tent map is rational if and only if its 
turning point is eventually periodic.

In this paper we prove that  the 
characterization above does not hold for maps with more than one turning 
point (multimodal case). We present families of polynomial maps where the
Artin-Mazur zeta function is the same rational function for every element.
These families contain maps that 
are hyperbolic, post-critically finite, and non-hyperbolic maps 
where all but one of the critical points have a cantor set as 
$\omega-$limit.



\subsection{Statement of the results}
We recall some definitions and set notation before stating the results.
Let $I \subset \R$ be a compact interval with non empty interior. 
A continuous map 
$f \colon I \longrightarrow I$ is \emph{piecewise monotone} if the interval
$I$ can be subdivided into finitely many (maximal) subintervals on which $f$
is alternately strictly increasing or strictly decreasing. If there are only
two such subintervals, we say that $f$ is a \emph{unimodal map} or just 
\emph{unimodal}. The boundary points of these monotonicity intervals that
belong to the interior of $I$ are called the \emph{turning points} of $f$.

A point $x \in I$
is called $\emph{periodic}$ if there exits an integer $n \geq 1$ such that 
$f^n(x) = x$. The smallest such $n$ is called the $\emph{minimal period}$ of
$x$, and the finite set \[ \cO_f(x) \= \{x, f(x), \ldots f^{n-1}(x) \}, \] is 
called the \emph{periodic cycle} or \emph{periodic orbit} of $x$. The 
period of $\cO_{f}(x)$ is $n$. 
The periodic point $x$ (or the cycle
$\cO_f(x)$) is called \emph{repelling} provided there is an open set $U$
containing $x$ such that, for every $y \in U \setminus \{x \}$, there is 
a $k(y) \in \N$ such that $f^{nk(y)}(y) \notin U$. 
We say that a point $y \in I$
is \emph{asymptotic} to the periodic point $x$ (or the cycle $\cO_f(x)$) if 
there exists $m(y) \in \N$ such that 
$\lim_{k \to \infty} f^{kn + m(y)}(y) = x$.

Fix $d \geq 3$.
Let $1 \leq \nu \leq d-1$, and 
$\boldsymbol{\mu} = (\mu_1,\mu_2, \ldots, \mu_{\nu}) \in \Z_+^{\nu}$, with 
$\sum_{i = 1}^{\nu}\mu_i = d-1$. We denote by $\cP_d^{\boldsymbol{\mu}}(\R)$
the set of real polynomials of degree $d$ with $\nu$ distinct real critical
points $c_1, c_2, \ldots, c_{\nu}$ and corresponding order 
$\mu_1, \mu_2, \ldots, \mu_{\nu}$. It is known that 
$\cP_d^{\boldsymbol{\mu}}(\R)$ is a smooth real manifold of dimension $\nu$
and the critical values form a full set of coordinates. See for example
\cite{Milnor_Tresser_Entropy_monotonicity_for_real_cubic, Levin_Shen_van_Strien_Transversality_elementary_proof}.

\begin{namedtheorem*}{Main Theorem}
    \label{theo:AM_zeta_function_for_polynomials}
    For every $d \geq 3$, every $2 \leq \nu \leq d-1$, and every 
    $\boldsymbol{\mu} = (\mu_1,\mu_2, \ldots, \mu_{\nu}) \in \Z_+^{\nu}$ with
    $\sum_{i = 1}^{\nu}\mu_i = d-1$, there is a family 
    $\cU(\nu, \boldsymbol{\mu})$ of codimension one in 
    $\cP_d^{\boldsymbol{\mu}}(\R)$ such that every map in this family has 
    the same rational Artin-Mazur zeta function. Moreover, all but one of
    the critical orbits undergoes independent bifurcations. 
\end{namedtheorem*}
We are also able to say a bit more about these families. Recall that two 
piecewise monotone maps $f$ and $\hat{f}$ with $\ell \geq 1$ turning points
$c_1 < c_2 < \ldots < c_{\ell}$ and 
$\hat{c}_1 < \hat{c}_2 < \ldots < \hat{c}_{\ell}$ respectively are called
\emph{combinatorially equivalent} or that they are of the 
\emph{same combinatorial type} if there exists an order preserving 
bijection \[ h \colon \bigcup_{i=i}^{\ell} \cO_{f}(c_i) \longrightarrow 
\bigcup_{i=1}^{\ell} \cO_{\hat{f}}(\hat{c}_i),\] such that 
$h \circ f = \hat{f} \circ h$, and for $i = 1,2, \ldots, \ell$ 
we have that $h(c_i) = \hat{c_i}$. 

\begin{theo}
    \label{theo:combinatorial_type_theorem}
    The family $\cU(\nu, \boldsymbol{\mu})$ contains the following type 
    of maps:
    \begin{enumerate}
        \item[(i)] Hyperolic maps maps of infinitely many different
        combinatorial type;
        \item[(ii)] Post-critically finite maps of infinitely many different 
        combinatorial type;
        \item [(iii)] Non-hyperbolic maps that are not post-critically finite
        of infinitely many different combinatorial type.
    \end{enumerate} 
    In particular, the family $\cU(\nu, \boldsymbol{\mu})$ contains 
    infinitely many different conjugacy classes of polynomial maps.
\end{theo}

\begin{rema}
    In \cite{Kozlovski_2019_estructure_of_isentropes} Kozlovski raised the 
    question if in the space of real polynomials of degre $d > 2$ with real
    critical points there exist isentropes (this are sets where the entropy
    is constant) of positive entropy containing hyperbolic maps of 
    infinitely many different combinatorial types 
    (see \cite[Question 2]{Kozlovski_2019_estructure_of_isentropes}). In 
    Theorem \ref{theo:combinatorial_type_theorem}, the different 
    combinatorial type of the hyperbolic class come from changing the period
    of a periodic critical point and making small perturbations on a special
    class of maps introduces in Section \ref{section_virtually_unimodal_maps}.
    In particular, they do not belong to the same isentrope. 

    We do believe that the arguments in the proof of Main Theorem and in 
    Theorem \ref{theo:combinatorial_type_theorem} can be adapted to obtain 
    hyperbolic map of different combinatorial type without changing the 
    period of the attracting cycle, and thus not changing the entropy along 
    the perturbations.
    
\end{rema}

We include a proof of the dichotomy for the rationality of the 
Artin-Mazur zeta function in the unimodal case. The most general case 
only need a condition on the number non-repelling periodic points of the map.

\begin{theo}[\cite{MiTh88}]
\label{theo:Theorem_1}
Let $I \subset \R$ be a compact interval, and 
$f \colon I \longrightarrow I$ a unimodal map. 
Suppose that all but finitely many of the periodic points of 
$f$ are repelling. Then, the Artin-Mazur 
zeta function of $f$ is a rational function if and only if the turning
point of $f$ is asymptotic to a periodic orbit of $f$.
\end{theo}

The condition that all but finitely many of the periodic points are repelling
is needed to avoid maps like the tent map 
$T_1(x) \= (1-|x|)-1$ from $[-1,1]$ to itself, that has an interval of 
neutral fixed points. Less trivial examples where developed by Kaloshin and 
Kozlovski \cite{Kaloshin_Kozlovski_Arbitraty_fast_growth_of_the_number_of_periodic_points}. They presented a $C^r$ unimodal map $f$ whose number of isolated
periodic points 
$N_n(f)$ grows faster than any given sequences along a sebsequence of 
$n_k=3^k$. Their construction involve the presence of infinitely many 
indifferent periodic points.


Martens, de Melo, and van Strien \cite{Martens_de_Melo_van_Strien_Julia_Fatou_Sullivan} proved that $C^r$
unimodal maps, for $r \geq 2$, with non-flat critical points and finitely 
many periodic points of each period satisfy this condition (see \cite[Theorem B]{Martens_de_Melo_van_Strien_Julia_Fatou_Sullivan}).



\subsection{Organization}
We now describe the organization of the paper.

In Section \ref{Section_Preliminaries} we review some general theory and 
results that we will need in the spirit of keep this paper self content. 
We start reviewing some basic definitions in the theory of automatic
sequences and its connection with the rationality of power series,
that we will use for the proof of Theorem \ref{theo:Theorem_1}. We also 
present basic constructions in the kneading theory of Milnor and Thurston,
describing the relationship between the forward orbit of turning points, 
encoded in a power series called kneading determinant, and 
the Artin-Mazur zeta function for piecewise monotone maps. This result will
allow us to reduce the study of the rationality of the Artin-Mazur zeta
function to the study of the rationality of the kneading determinant.

In Section \ref{section:unimodal_case}, we present the proof of 
Theorem \ref{theo:Theorem_1}. We give a description of the kneading
determinant in the case of unimodal maps, and present some general results
in the characterization of algebraic elements in the ring of formal power 
series with coefficients in a finite set of integers (Proposition 
\ref{coro:transcendence_criteria}). We use these results to give a 
characterization of the rationality of the kneading determinant in the 
unimodal case (Proposition \ref{prop:Proposition_1}).

In Section \ref{sec:Example}, we recall some basic notions and results about
shift spaces. In particular, we focus on the Fibonacci shift that we will
use to construct explicit examples of the families mentioned in the Main
Theorem. We include and 
explicit computation of the Artin-Mazur zeta function for the Fibonacci
shift.


In Section \ref{section_virtually_unimodal_maps} we introduce the concept 
of virtually unimodal maps (see Definition \ref{def:VUM}). This type of maps
will play a central role in our work as we will use them as starting
points to construct the codimension one families mentioned in the Main 
Theorem. We will describe their 
combinatorics when the orbit of each turning point is finite, and give an
explicit method to construct these virtually unimodal combinatorics. 

Section \ref{section:Proof_of_the_main_theorem} we give the proof of the 
Main Theorem and Theoreom \ref{theo:combinatorial_type_theorem}. We start by
proving that the virtually unimodal combinatorics,
introduced in Section \ref{section_virtually_unimodal_maps}, can be 
realized by post-critically finite polynomials. Then, for a prescribed 
virtually unimodal polynomial, we study the structure of its repeller and 
how this deform under small perturbation on a suitable space. Finally, 
combining the previous results, we prove the Main Theorem using a
transversality result for the critical relations described by our starting
post-critically finite polynomial. The proof of Theorem \ref{theo:combinatorial_type_theorem} is a direct consequences of the 
transversality result used in the proof of the Main Theorem together with 
the fact that prescribed virtually unimodal maps can be choose of different 
combinatorial type.

Finally, in Section \ref{section_the_bimodal_case}, we present an explicit
construction for the cubic case.







\subsection{Acknowledgements}
The author would like to thank Juan Rivera-Letelier and Weixiao Shen for
helpful discussions during the preparation of this work. 


\section{Preliminaries}
\label{Section_Preliminaries}
 
Throughout the rest of this work, we will use $\N$ to denote the set of
positive integers and $\N_0$, to denote the set of non-negative integers,
thus $\N_0 \= \N \cup \{ 0 \}$.

Given $a, b \in \R$, we denote by $\langle a, b\rangle$ the convex hull of 
$\{ a,b \}$. Thus, in the case where $a \neq b$, the set
$\langle a, b\rangle$ is the smallest closed interval whose
boundary is equal to the set $\{ a,b\}$.


\subsection{Automatic sequences}
\label{subsec:automatic_sequences}
We present some general definitions and results in the context of our work.
We refer the reader to \cite{Allouche_Shalit2003_Automatic_sequences} for a 
detailed and comprehensive presentation of this material.

Given a finite set $A$, a \emph{finite word} over $A$ is a finite 
string of elements of $A$. The \emph{length} of a finite word $w$ 
is the number of elements of $A$ in it, and it is denoted by $|w|$.
We denote by $A^n$ the set of all word over $A$ of length $n$, 
thus \[ A^n \= \{ w = w_0w_1w_2 \ldots w_{n-1} \colon w_i \in A
 \text{ for every } i = 0,1,2, \ldots,n-1 \}. \]
Put $A^* \= \cup_{n \geq 0} A^n $, thus $A^*$ is the set of 
all finite words over the set $A$, where 
$A^0 = \{ \e \}$, and $\e$ represent the empty word. The
length of the empty word is zero.

The set $A^*$ together with the concatenation operation is a monoid 
with identity element $\e$,
usually called the \emph{free monoid} generated by $A$. We will use
the multiplicative notation for the concatenation operation, thus
if $w, u \in A^*$ with $w = w_0w_1 \ldots w_{n-1}$ and 
$u = u_0u_1 \ldots u_{m-1}$, then 
$w u = w_0w_1 \ldots w_{n-1}u_0u_1 \ldots u_{m-1}$. With this 
notation, for $i \in \N$, we have that $w^i$ is the word generated
by the concatenation of $w$ with itself $i$ times.

An \emph{infinite word} (or sequence, we will use both terms 
interchangeably) over the set $A$ is a map from $\N_0$ to $A$,
denoted by $\bm{w} = w_{0}w_{1}w_{2}w_{3} \ldots$ (we will also
use the sequence notation $\bm{w} = (w_n)_{n \in \N_0} = 
(w_n)_{n \geq 0}$ and we can use any infinite subset of $\N_0$ to index our 
infinite words). We will denote the set of all infinite words over $A$
by $A^{\N_0}$.
We often use bold letters to denote elements in a set of 
infinite words.
If $w \in A^*$ is a nonempty word, we denote
by $\overline{w}$ the infinite word generated by infinite 
concatenations of the word $w$ with itself. An infinite word
$\bm{w}$ is said to be \emph{periodic}, of period $n \in \N$ if there
is a $w \in \cA^n$ such that $\bm{w} = \overline{w}$.
If there are $w, u \in A^*$ such that 
$\bm{w} = u \overline{w}$, then we say that $\bm{w}$ is 
\emph{eventually periodic} or \emph{preperiodic} of period $|w|$ and 
preperiod $|u|$.

Informally speaking, a sequence $\bm{w} = (w_n)_{n \geq 0} \in A^{\N_0}$ is
called \emph{$k$-automatic} for an integer $k \geq 2$, if there exists a 
deterministic finite automaton with output ($k$-DFAO in short) with input 
language $\{0,1, \ldots, k-1 \}$ that when fed with the base $k$ 
representation of $n$, starting with the most significant digit, returns
$w_n$ (see \cite[Section 5.1]{Allouche_Shalit2003_Automatic_sequences} for a
precise definition). 

The following theorem gives a characterization of automatic sequences in 
terms of uniform morphisms. Recall that given two finite sets $A$ and $B$
a \emph{uniform morphism} from $A^*$ to $B^*$ 
of length $k$ is a (monoid) morphism $\phi \colon A^* \longrightarrow B^*$
such that the image of each $a \in A$ is a word of length $K$, thus 
$\phi(A) \subset B^k$. A $1$-uniform morphism is called a \emph{coding}.

\begin{theo}[Cobham's Theorem 1]
    \label{thm:Cobham_theorem_1}
    Let $k \geq 2$. A sequence $\bm{w} = (w_n)_{n\geq 0}$ is $k-$automatic
    if and only if it is the image, under a coding, of a fixed point of a 
    $k$-uniform morphism.
\end{theo}


\begin{theo}[Cobham's Theorem 2]
Let $k,l \geq 2$ be multiplicatively independent integers. If the 
sequence $\bm{w} = (w_n)_{n \geq 0}$ is both $k-$automatic and 
$l-$automatic, then $\bm{w}$ is eventually periodic.
\end{theo}


\subsection{Formal power series}
\label{subsec_Formal_power_series}

Let $K$ be a field. The ring of \emph{formal power series} over $K$
consist of formal infinite sums 
\[ f(x) = \sum_{n \geq 0} a_nx^n, \] with $a_n \in K$ for all 
$n \geq 0$, denoted by $K[[x]]$, with the usual operations, addition 
term by term and the formal series product (Cauchy product).

A \emph{formal Laurent series} over $K$ is a formal infinite sum
\[ f(x) = \sum_{n \in \Z} a_n x^n, \]with $a_n \in K$ for every 
$ n \in \Z$, such that for only finitely many $n < 0$ we have that 
$a_n \neq 0$. We denote the set of all formal Laurent series over 
$K$ by $K((x))$, it is a ring  and can be identify with the  fraction
field  of $K[[x]]$. 

We can include the polynomial ring $K[x]$ into the formal power series
ring $K[[x]]$. Since $K((x))$ is a field, this extends to an inclusion
of the field of rational function $K(x)$ into the formal Laurent series
field $K((x))$. 

The following theorem gives an explicit description of the algebraic 
elements of $K[[x]]$ over $K(x)$ when $K$ is a finite field, in terms
of the sequence of coefficients.

\begin{theo}[Christol's Theorem]
Let $A$ be a nonempty finite set, and let $\bm{w} = (w_i)_{i \geq 0}$ be a 
sequence over $A$. Let $p$ be a prime number. Then $\bm{w}$ is $p $-automatic
if and only if there exists an integer $n \geq 1$ and an injective map
$\rho \colon A \longrightarrow \F_{p^n}$ such that the formal power series
$\sum_{i \geq 0} \rho(w_i)x^i$ is algebraic over $\F_{p^n}(x)$.
\end{theo}

\subsection{Kneading theory of piecewise monotone maps} \label{subsec_multimodal_Kneading_thoery}

Now we will present some of the basic constructions of the kneading 
theory developed by Milnor and Thurstone in \cite{MiTh88}.

Let $I = [a, b]$ be a compact interval of real numbers. A continuous map
$f \colon I \longrightarrow I$ is called \emph{piecewise monotone} if there
exists a finite collection of points 
\[a = c_0 < c_1 < \ldots < c_m < c_{m+1} = b,\] such that the restriction 
of $f$ to any interval $I_i \= [c_i,c_{i+1}]$, with $0 \leq i \leq m$, is
strictly monotone, and each of the $I_i$ is a maximal with this property.
We also call these map $m-$\emph{modal}. In the particular case when
$m$ equals one we call the map \emph{unimodal}, and when $m$ equals two 
we call it \emph{bimodal}.
Observe that each of the points $c_0, c_1, \ldots , c_{m+1}$ is either
a local maximum or a local minimum of $f$. We call the interior local 
extremum the \emph{turning points} of $f$.
We will denote the set of turning points of $f$ by $\Trn (f)$.

Given 
a vector $s = (s_0, s_1, \ldots, s_{m})$ of $m+1$
alternating signs, thus $s_j = (-1)^j$ or $(-1)^{j+1}$, we will say
that a $m-$modal map $f \colon I \longrightarrow I$ has 
\emph{shape $s$} if the restriction of $f$ to $I_i$ is monotone 
increasing or monotone
decreasing accord as $s_i$ is equal to $+1$ or $-1$ respectively.

For any $x \in I$, we can 
assign an address $A(x) $ in the set  
\[ \cD = \cD(m) \= \{ I_0, C_1, I_1, C_2, \ldots, I_{m-1}, C_m, I_{m} \}\]
that is equal to the formal symbol $I_i$ if 
$x \in I_i \setminus \{ c_1, \ldots , c_m \}$ for some $0 \leq i \leq m$,
or the formal symbol $C_j$ if $x = c_j$ for some $1 \leq j \leq m$. 
Then the itinerary of the point $x \in I$ is the sequence 
\[ \cI(x) \= (A_0(x),A_1(x) ,A_2(x), \ldots, ), \] where 
$A_n(x) \= A(f^n(x)).$ For each symbol in $\cD$ we define its
\emph{sing} $\e(I_i) = s_i$ and $\e(C_i) = 0$.


Since $f$ is continuous, for each $n \geq 0$, we can 
find some $\delta_n >0$ such that for every $y \in (x, x+ \delta_n)$ 
the address $A_n(y)$ is constant, and for every $0 \leq j < n$ we have
that $\e(A_j(y))$ is nonzero. We denote this address by $A_n(x^+)$ and 
its sign by $\e_n(x^+)$. In the same way we can define $A_n(x^-)$ and 
$\e_n(x^-)$ considering $y \in (x - \delta_n , x)$.

Now, for each turning point $c_i$ we can construct the formal power series
\[ \theta(c_i^+) \= \sum_{n \geq 0} \theta_n(c_i^+)t^n, \]
and
\[ \theta(c_i^-) \= \sum_{n \geq 0} \theta_n(c_i^-)t^n, \]
where the coefficients are defined by 
$\theta_0(c_i^{\pm}) \= A_0(c_i^{\pm})$ and for $n \geq 1$
\[ \theta_n(c_i^{\pm}) \= \e_0(c_i^{\pm}) \e_1(c_i^{\pm}) \ldots \e_{n-1}(c_i^{\pm}) A_n(c_i^{\pm}). \] Each of these formal power series have 
coefficients in the set $ \{ \pm I_0, \pm I_1, \ldots, \pm I_{m} \}$. 

\begin{defi}
The difference $\theta(c_i^+) - \theta(c_i^-)$, evaluated at the turning
point $c_i$, will be called the $i$th \emph{kneading increment} of $f$.
We will denote it by $\nu_i$.
\end{defi}

The kneading increments can be written as a linear combination of the 
formal symbols $\{ I_0, I_1, \ldots , I_{m} \}$ over the ring 
$\Q[[t]]$. Thus, for each turning point $c_i$ there are 
$N_{c_i0}, N_{c_i1}, \ldots, N_{c_im} \in \Q[[t]]$ such that 

\[ \nu_i = N_{c_i0}I_0 + N_{c_i1}I_1 + \ldots + N_{c_im}I_m.\] Then,
the \emph{kneading matrix} of $f$ is the $m \times (m+1)$ matrix 
$[N_{c_ij}]$.

Let $D_i$ be the $m \times m$ matrix obtained by deleting the $i$th 
column of the kneading matrix $[N_{c_ij}]$. Then the ratio

\begin{equation}
    \label{eq:kneading_determinant_formula}
   D(t) \= \frac{(-1)^{i+1} \det (D_i)}{1 - \e (I_i)t}, 
\end{equation}
is a fixed
element of $\Z[[t]]$, independent of the choice of $i$, whose leading 
coefficient is $1$ (hence is a unit of $\Z[[t]]$), see 
\cite[Lema 4.3]{MiTh88}. We call $D(t)$ the \emph{kneading determinant}
of $f$. Milnor and Thurston gave an explicit relationship between the 
kneading determinant of a piecewise monotone map and its Artin-Mazur zeta 
function.

\begin{theo}[\cite{MiTh88}]
\label{theo:Milnor_Thurston_kneading_determinant_zeta_function_theorem}
Let $f \colon I \longrightarrow I$ be a piecewise monotone map such 
that all but finitely many of its periodic points are unstable. Then 
\[ \frac{1}{\zeta_f(t)} = \Phi(t) D(t), \] where $\Phi(t)$ is a product 
of cyclotomic polynomials.
\end{theo}

Here, by cyclotomic polynomial, we mean any product of finitely many factors
of the form $1-t^p$.

\subsection{Hyperbolic geometry}
\label{subsec_hyperbolic_geomtry}

We will understand as the \emph{hyperbolic plane} the unit disk $\D$ with the 
\emph{hyperbolic metric} 
\begin{equation}
    \label{eq:hyperbolic_metric}
    \lambda_{\D}(z)|dz| \= \frac{2|dz|}{1 - |z|^2}.
\end{equation}
This metric induces a \emph{hyperbolic distance} in the following way. For 
$z,w \in \D$ put 
\begin{equation}
    \label{eq:hyperbolic_distance}
    d_{\D}(z,w) \= \inf_{\gamma} \int_{\gamma} \lambda_{\D}(z)|dz|,
\end{equation}
where the infimum is taken over all smooth curves $\gamma$ joining $z$ with $w$
in $\D$. The integral \[\int_{\gamma} \lambda_{\D}(z)|dz|,\] is called the 
\emph{hyperbolic length} of $\gamma$.


The distance can be computed explicitly

\begin{theo}
    \label{theo:euclidean_hyp_distance}
    For every $z,w \in \D$ the hyperbolic distance $d_{\D} (z,w)$ is given by 
    \[ d_{\D}(z,w) = \log \frac{1 + \left|\frac{z-w}{1-z\ov{w}} \right| }{1-\left|\frac{z-w}{1-z\ov{w}} \right|}. \]
\end{theo}
In particular, we get that 
\begin{equation}
    \label{eq:euclidean_hyp_dist}
    |z-w| \leq e^{d_{\D}(z,w)} -1.
\end{equation}

The Riemann Mapping Theorem enables us to transfer the hyperbolic metric from
$\D$ to any simply connected proper domain $D \subset \C$.

\begin{defi}
    \label{def:hyp_metric_on_domains}
    Suppose that $f$ is a conformal map of a simply connected domain 
    $D \subset \C$ onto $\D$. Then the hyperbolic metric $\lambda_{D}(z)|dz|$
    of $D$ is defined by \[\lambda_D(z) = \lambda_{\D}(f(z))|f'(z)|.\] The 
    hyperbolic distance $d_D$ is the distance function on $D$ derived from the 
    hyperbolic metric $\lambda_D(z)|dz|$.
\end{defi}
This definition is independent of the choice of the conformal  map $f$, thus 
$\lambda_D$ is determined by $D$ alone 
(see \cite{Beardon_Minda_2007_hyp_metric}). In particular, for any conformal 
map $f$ from $D$ onto $\D$ we have 
\begin{equation}
    \label{eq:hyp_metric_on_domains}
    d_{\D}(z,w) = d_{\D}(f(z),f(w)) 
\end{equation}
\begin{theo}[Pick Theorem]
    \label{theo:Pick_theorem}
    If $f \colon D \longrightarrow D'$ is a holomorphic map between simply 
    connected domains, then exactly one of the following statements is valid:
    \begin{enumerate}
        \item[i)] $f$ is a conformal  isomorphism from $D$ onto $D'$, and 
        maps $D$ with its hyperbolic metric isometrically onto $D'$ with its
        hyperbolic metric.
        \item[ii)] $f$ strictly decreases all nonzero distances. In fact, for
        any compact $K \subset D$ there is a constant $M_K < 1$ such that 
        for every $z,w \in K$ we have 
        \[ d_{D'}(f(z),f(w)) \leq M_K \, d_{D}(z,w). \]
    \end{enumerate}
\end{theo}

The following lemma is a direct consequence of Pick Theorem.
\begin{lemm}
    \label{lemm:uniform_contraction}
    Let $f_0, f_1, \ldots , f_n$ be holomorphic functions on a simply connected
    domain $D$ and continuous on $\overline{D}$ such that for every 
    $i \in \{ 0,1, \ldots , n\}$, $f_i(\overline{D}) \subset D$. Let 
    $(w_i)_{i \geq 0} \in \{ 1, 2, \ldots ,n \}^{\N_0}$ and consider the 
    sequence $\{ F_k \}_{k \geq 0}$ where \[F_k \= f_{w_0} \circ f_{w_1} \circ 
    \ldots \circ f_{w_k}.\] Then $\{ F_k \}_{k \geq 0}$ converges uniformly on 
    $D$ to a constant function. 
\end{lemm}

\begin{proof}
Since for every $i \in \{ 0,1, \ldots, n \}$ we have that 
$f_i(\ov{D}) \subset D $. By Pick Theorem, none of the $f_i$ can be a conformal
isomorphism. Moreover, for each $i,j$ in $\{ 0,1, \ldots, n \}$ there is a 
constant $M_{j,i} < 1$ such that for all $z,u \in D$ we have 
\begin{equation}
    \label{eq:first_inequality}
    d_D(f_j(f_i(z)),f_j(f_i(u))) \leq M_{j,i}d_D(z,w).
\end{equation}

Let $(w_i)_{i \geq 0} \ in \{ 0,1, \ldots, n \}^{\N_0}$. For every $k \geq 0$,
using \eqref{eq:first_inequality} inductively, we get that for all 
$z, u \in D$
\begin{align}
    \label{eq:second_inequality}
    d_D(F_{k+1}(z),F_{k+1}(u)) &= d_D(f_{w_0} \circ \ldots \circ f_{w_k}(z),
    f_{w_0} \circ \ldots \circ f_{w_k}(u)) \\ \nonumber
    &\leq M_{w_0w_1}d_D(f_{w_1} \circ \ldots \circ f_{w_k}(z), 
    f_{w_1} \circ \ldots \circ f_{w_k}(u))\\ \nonumber
    &\vdots \\\nonumber
    &\leq M_{w_0w_1}M_{w_1w_2} \ldots M_{w_{k-1}w_k}d_D(f_{w_{k+1}}(z),
    f_{w_{k+1}}(u)).
\end{align}
Put $M \= \max \{ M_{ji} \colon i,j \in \{ 0,1, \ldots, n \} \}$. Since 
$f_i(\ov{D})$ is compact for all $i\in \{ 0,1, \ldots, n \}$, we can set
\[\alpha \= \max\{ \diam_D(f_i(\ov{D})) \colon i\in \{ 0,1, \ldots, n \} \}.\] 
Then, for all $k \geq 0$ and all $z,u \in D$
\begin{equation}
    \label{eq:third_inequality}
    d_D(f_{w_{k+1}}(z), f_{w_{k+1}}(u)) \leq \alpha.
\end{equation}
Using \eqref{eq:second_inequality} and \eqref{eq:third_inequality}, we get
\begin{equation}
    \label{eq:fourth_inequality}
    \diam_D(F_{k+1}(D)) \leq \alpha \prod_{i=0}^{k-1}M_{w_iw_{i+1}} \leq 
    \alpha M^{k}.
\end{equation}
Then $\diam_D(F_{k}(D)) \longrightarrow 0$ as $k \rightarrow \infty$.

Also, for every integers $0 \leq k < n$ we have 
\[F_n(D) = F_k(f_{k+1} \circ \ldots f_n(D)) \subset F_k(D).\] This implies that 
for every $x \in D$
\[d_D(F_{n}(x), F_k(x)) \leq \diam_D(F_k(D)).\] Fix $\e > 0$. Letting $N \geq 0$
be large enough so that $\diam_D(F_k(D)) < \e$ for all $k \geq N$, we get 
that for every $x \in D$ and every $n \geq k \geq n \geq N$
\[d_D(F_n(x), F_k(x)) \leq \diam_D(F_k(D)) < \e.\] Thus, $\{ F_k\}_{k \geq 0}$
is uniformly Cauchy on $D$. In particular, $\{ F_k\}_{k \geq 0}$ converges
uniformly on $D$ to a holomorphic function $F$. 

Finally, observe that for every $z,u \in D$ and all $k \geq 0$, by 
\eqref{eq:euclidean_hyp_dist}, and \eqref{eq:hyp_metric_on_domains}
\begin{align*}
    |F(z) - F(u)| &= \lim_{k \to \infty}| F_k(z) - F_k(u) | \\
    &\leq \lim_{k \to \infty} \left( e^{d_D(F_k(z),F_k(u))} -1 \right) \\
    &\leq \lim_{k \to \infty} \left( e^{\diam_D(F_k(D))} -1 \right) \\
    &= 0.
\end{align*}
This implies that $F$ must be a constant function on $D$.

\end{proof}

For an interval $J \subset \R$
let $\C_J \= \C \setminus (\R \setminus J)$ denote the complex plane slit 
along the two rays that are the complement of $J$ in $\R$. 
A \emph{geodesic neighborhood} of $J$ of angle 
$\theta \in (0, \pi)$ is the union of two $\R-$symmetric segments
of euclidean disks based on $J$ and having angle $\theta$ with $\R$.
We will need the following version of the Schwarz Lemma (see 
\cite{Lyubich-Yamposlky97}):

\begin{lemm}[Schwarz Lemma]
    \label{lemm:Schwarz_Lemma}
    Let $J \subset \R$ and $J' \subset \R$ be two intervals. Let 
    $\phi \colon \C_J \longrightarrow\C_{J'}$ ba an analytic map such that 
    $\phi(J) \subset J'$. Then for any $\theta \in (0, \pi)$, we have 
    that $\phi(D_{\theta}(J)) \subset D_{\theta}(J')$.
\end{lemm}


\section{The unimodal case} 
\label{section:unimodal_case}

In view of Theorem \ref{theo:Milnor_Thurston_kneading_determinant_zeta_function_theorem}, we have that Theorem \ref{theo:Theorem_1} is a direct
consequence of the following proposition

\begin{prop}
\label{prop:Proposition_1}
Let $I \subset \R$ be a compact interval and let 
$f \colon I \longrightarrow I$ be a unimodal map with all but finitely many
of its periodic points repelling. 
Then, the kneading determinant of $f$ is a rational
function if and only if the turning point of $f$ is asymptotic to a 
periodic cycle of $f$.
\end{prop}

In order to prove Proposition \ref{prop:Proposition_1},we will need to 
understand the kneading determinant, introduced in Section 
\ref{subsec_multimodal_Kneading_thoery}, as a formal power series and as a 
dynamical object determine by the orbit of the unique turning point.

\subsection{The kneading determinant of unimodal maps}
In the case of unimodal maps, the kneading determinants become simpler.

Let $I \subset \R$ be a compact interval and
let $f \colon I \longrightarrow I$ be a unimodal map with turning point
$c$. Following Section \ref{subsec_multimodal_Kneading_thoery}, we have 
that
\[ I_0 \= \{ x \in I \colon x<c \} \hspace{1cm} \text{ and } \hspace{1cm} 
 I_1 \= \{ x \in I \colon x > c \}. \] 
 For $n \geq 1$ we put $\e_{n} \= \e(A_n(c))$ if $f^n(c) \neq c$, and 
 $\e_n = \e_1 \e_2 \ldots \e_{n-1}$ if $f^n(c) = c$. Then, the kneading 
 determinant of $f$ is the power series 
 \[ D_f(t) \= 1 + \e_1t + \e_1\e_2t^2 + \e_1\e_2\e_3t^3 + \ldots  \]

\begin{rema}
    This definition of the kneading determinant for unimodal maps coincide
    with the general definition given in Section \ref{subsec_multimodal_Kneading_thoery}, 
    see Lemma 4.5 in \cite{MiTh88}.
\end{rema}

Now, we will study the kneading determinant as a formal power series. For
that purpose, we will need the following general proposition.

Given a formal power series $f(x) = \sum_{i \geq 0}a_ix^i \in \Z[[x]]$, and a
prime number $p$, we define the \emph{reduction modulo $p$ of $f(x)$} 
as the formal power
series \[ [f(x)]_p \= \sum_{i \geq 0} (a_i \mod{p}) x^i \in \F_{p}[[x]]. \]
\begin{prop}
\label{coro:transcendence_criteria}
If the coefficients of $f(x) \in \Z[[x]]$ belong to a finite set, then
either $f(x) \in \Q(x)$ or $f(x)$ is transcendental over $\Q(x)$. In the 
former case, the coefficients of $f(x)$ form an eventually periodic sequence.
\end{prop}

\begin{proof}
    Let $f(x) \in \Z[[x]]$. We need to prove that if $f(x)$ is algebraic over 
    $\Q(x)$, then $f(x) \in \Q(x)$. Suppose that $f(x)$ is algebraic over 
    $\Q(x)$. First, we will prove that there are 
    $G_0(x), G_1(x), \ldots, G_d(x) \in \Z[x]$, not all identically zero,
    such that 
    \[G_0(x) + G_1(x)f(x) + \ldots + G_d(x)(f(x))^d = 0.\] Since $f(x)$ is
    algebraic over $\Q(x)$, there exist $R_0(x), \ldots , R_d(x) \in \Q(x)$,
    not all identically zero,
    such that \[ \prod_{i=0}^d R_i(x) (f(x))^i = 0.\] Write 
    $R_i(x) = \frac{P_i(x)}{Q_i(x)}$, with 
    \[P_i(x) = a_{i,0} + a_{i,1}x + \ldots a_{i,m_i}x^{m_i} 
    \hspace{0.3cm} \text{ and } \hspace{0.3cm} 
    Q_i(x) = b_{i,0} + b_{i,1}x + \ldots b_{i,n_i}x^{n_i},\] in $\Z[x]$.
    If $R_i(x)$ is the constant polynomial equal to zero, we put $P_i(x) = 0$
    and $Q_i(x) = 1$.
    Put \[ H(x) \= \prod_{i=0}^d Q_i(x)
    \hspace{0.4 cm} \text{ and } \hspace{0.4cm} 
    G_i(x) \= H(x) R_i(x), \] for all $i = 1,2, \ldots, d$
    and note that $G_i(x) \in \Z[x]$, they are not all identically zero, as 
    if $G_i(x) \neq 0$, then $K(x)G_i(x)$ is nonzero, and 
    \begin{align*}
        G_0(x) + G_1(x)f(x) + \ldots + G_d(x)(f(x))^d
        &= H (x) \prod_{i=0}^d R_i(x) (f(x))^i  \\
        &= 0.
    \end{align*}

    Denote by $C$ the set of all coefficients of the polynomials 
    $G_0(x), \ldots, G_d(x)$ together with the coefficients of $f(x)$. Note 
    that $C$ is a finite set of integers, then we can find infinitely many
    prime numbers not dividing any element of $C$. For any of these prime
    numbers, we have that $[f(x)]_p$ is not identically zero, and for any 
    $G_i(x)$ not identically zero, $[G_i(x)]_p$ is not identically zero.
    In particular
    \[ [G_0(x)]_p + [G_1(x)]_p[f(x)]_p + \ldots [G_d(x)]_p ([f(x)]_p)^d 
    = 0, \] with not all of the $[G_i(x)]_p$ identically zero.
    So, $[f(x)]_p$ is algebraic over $\F_p(x)$. In particular, 
    by Christol's Theorem, 
    we can find two primes $p_1 \neq p_2$ such that the sequence 
    of coefficients of $f(x)$ is $p_1$ and $p_2$-automatic. By Conham's 
    Theorem, we conclude that the sequence of coefficients of $f(x)$ form 
    an eventually periodic sequence.
\end{proof}

\subsection{Proof of Proposition \ref{prop:Proposition_1}}

\begin{proof}[Proof of Proposition \ref{prop:Proposition_1}]
Let $f \colon I \longrightarrow I$ be a unimodal map with turning point $c$.
For $i, k \in \N$ with $i \leq k$, put \[\bf{e}_{i,k} \= \prod_{j=i}^k \e_j.\]
Observe that, for every $p \geq 1$ we can write
\[D_f(t) = 1 + \bf{e} _{1,1}t + \ldots + \bf{e}_{1,p-2}t^{p-2} +
\bf{e}_{1,p-1}t^{p-1} (1 + \bf{e}_{p,p}t + \bf{e}_{p,p+1}t^2 + \ldots).\]
If $c$ is asymptotic to a periodic cycle, by \cite[Lemma 2.1]{MiTh88},
its itinerary is eventually periodic. Thus, there are integers $p \geq 0$ and 
$k \geq 1$ such that for every $n \geq 0$ 
\begin{equation}
    \label{eq:periodic_itinerary_unimodal}
    \e_{p+kn} = \e_p.
\end{equation}

By \eqref{eq:periodic_itinerary_unimodal}, the sequence  
$\{ \e_{p+i} \}_{i \geq 0}$ is periodic of period $k$, and for every 
$0 \leq j \leq k-1$ we have 
\[ \bf{e}_{p,p+k+j} = (\bf{e}_{p,p+j})^2 \bf{e}_{p+j+1,p+k-1},\] and
in particular, in the case when $j = k-1$, we have
\[ \bf{e}_{p,p+2k-1} = 1. \] This proves that for every $j \in \N_0$ we have
\[ \bf{e}_{p,p+2k+j} = \bf{e}_{p,p+j}. \] So, the sequence 
$ 1, \bf{e}_{p,p}, \bf{e}_{p,p+1}, \ldots, \bf{e}_{p,p+j}, \ldots $ is
periodic of period (not necessarily minimal) $2k$. Then
\[D_f(t) = 1 + \bf{e}_{1,1}t + \ldots + \bf{e}_{1,p-2}t^{p-2} +
\bf{e}_{1,p-1}t^{p-1} \left( \frac{1 + \bf{e}_{p,p }t + \ldots + 
\bf{e}_{p,p+2k-2} t^{2k-1}}{1 - t^{2k}} \right).\]
So, the kneading determinant of $f$ is a rational function. 

Suppose now that $D_{f}(t)$ is a rational function.
By Corollary \ref{coro:transcendence_criteria}, the coefficients of $D_f(t)$
form an eventually periodic sequence. Thus, there are $n_0 \geq 0$ and 
$p \geq 1$ such that for every $0 \leq j < p$ and every $m \geq 0$ we have 
$\bf{e}_{1,n_0+j} = \bf{e}_{1,n_0+j+mp}$. 
It follows that 
\begin{equation*}
    \e_{n_0+j+1} = \frac{\bf{e}_{1,n_0+j+1}}{\bf{e}_{1,n_0+j}} \\
    = \frac{\bf{e}_{1,n_0+j+1+p}}{\bf{e}_{1,n_0+j+p}} \\
    = \e_{n_0+j+p+1}.
\end{equation*}
So the sequence $\{ \e_i \}_{i \geq 1}$ is eventually periodic.

\end{proof}


\section{The Fibonacci shift} \label{sec:Example}

In this section, we will use the same definitions as in Section 
\ref{subsec:automatic_sequences}. Let $\cA$ be a finite set (\emph{alphabet}),
we will introduce a metric structure (then a topology) in the space 
$\cA^{\N_0}$, together with a transformation that will define a dynamical 
systems in our topological space.

Given $\bm{w} = (w_i)_{i \geq 0} \in \cA^{\N_0}$, and integers 
$0 \leq k \leq n$ we denote by 
$\bm{w}|_{[k,n]}$ the subword of $\bm{w}$ from $k$ to $n$, thus
\[\bm{w}|_{[k,n]} = w_kw_{k+1}\ldots w_{n}.\] In particular 
$\bm{w}|_{\{i\}} = w_i$, for every $i \in \N_0$. Two words 
$w_0w_1 \ldots w_n$ and $u_0u_1\ldots u_n$ in $\cA^{n+1}$ are different if 
there
is a $0 \leq i \leq n$ such that $w_i \neq u_i$. In the same way, we say that 
two sequences $\bm{w}$ and $\bm{u}$ in $\cA^{\N_0}$ are different if there
is an $i \geq 0$ such that $\bm{w}|_{\{i\}} \neq \bm{u}|_{\{i\}}$.

\begin{defi}
    The $\N_0-$\emph{full-shift $\cA^{\N_0}$} is the set of all infinite 
    sequences (\emph{configurations}) $\bm{w} = (w_i)_{i \geq 0}$, with the 
    prodiscrete topology, which can be described via the metric
    \[d_s(\bm{w}, \bm{u}) \= \sup(\{ 2^{-i} \colon \bm{w}|_{\{i\}} =
    \bm{u}|_{\{i\}}, i \in \N_0 \} \cup \{ 0 \}).\]
\end{defi}

Given a word $a_0 a_1\ldots a_n \in \cA^*$ and $k \in \N_0$ we define the
\emph{cylinder associated with this word at $k$} as the set 
\[[a_0a_1 \ldots a_n]_k \= \{ \bm{w} \in \cA^{\N_0} \colon \bm{w}|_{[k,k+n]}
= a_0a_1 \ldots a_n \}.\]
The collection of cylinder form a basis for the prodiscrete topology in 
$\cA^{\N_0}$.

With this topology, $\cA^{\N_0}$ is a Cantor set. Thus, $\cA^{\N_0}$ is 
compact, totally disconnected (the connected components are points), and 
perfect (has no isolated points.)

We define the shift map $\sigma \colon \cA^{\N_0} \longrightarrow \cA^{\N_0}$,
by $\sigma(\bm{w})|_{\{i\}} = \bm{w}|_{\{i+1\}}$ for all 
$\bm{w} \in \cA^{\N_0}$ and all $i \geq 0$. Thus, $\sigma (\bm{w})$ is the 
sequence resulting from deleting the first symbol of $\bm{w}$.

\begin{defi}
    A subshift consist of the pair $(X, \sigma)$ where $X$ is a closed $\sigma-$invariant subset of $\cA^{\N_0}$.
\end{defi}
We denote the subshift $(X, \sigma)$ by the set $X$. 
The \emph{Fibonacci shift} is given by 
\[ \SFib \= \{ \bm{w} \in \{ 1,2 \}^{\N_0} \colon \bm{w}|_{[i,i+1]}
\neq 11 \, \text{ for all } i \geq 0 \}. \] This subshift is the
collection of all infinite words in $\{1,2\}^{\N_0}$ that does not 
contain two consecutive $1'$s.




A finite directed graph $G$ consists of a finite collection of vertices 
$V(G) = \{ v_1,v_2, \ldots , v_k \}$ together with a finite collection of
directed edges $E(G) \subset V(G) \times V(G)$. Each edge 
$e = (v_i, v_j) \in E(G)$ is directed in the sense that has
an initial vertex $\text{in}(e) = v_i$ and terminal vertex $\text{ter}(e)=v_j$.

Let $G$ be a finite directed graph with $k$ labeled vertices. We define the 
$k \times k$ \emph{adjacency matrix} $A_G = (A_{i,j})$ of $G$ where 
$A_{i,j}$ is the number of directed edges with initial vertex $v_i$ and 
terminal vertex $v_j$.

Given a graph $G$, we define the \emph{edge shift} in the following way:
Label each edge in $G$ and consider as alphabet the set of labels $\cA_G$,
then
\[X_G \= \{ \bm{w} \in \cA_G^{\N_0} \colon \text{ter}(\bm{w}|_{ \{ i\}})
= \text{in}(\bm{w}|_{\{ i+1 \}}) \}.\]

The finite directed graph $G$ given in Figure \ref{fig:Fibonnaci_graph}
is so that $X_G = \SFib$.

Having this representation of SFT by finite directed graph has the utility that
one can work with the adjacency matrix. For instance, we have the following
proposition (see \cite[Proposition 2.2.12]{Lind-Marcus2021}). 

\begin{prop}
    \label{prop:periodic_pts_SFT}
    Let $A$ be a $k \times k$ matrix over $\N_0$, and let $m \geq 1$. The
    number of paths in $G_A$ of length $m$ from the vertex $v_i$ to the vertex
    $v_j$ is given by $(A^m)_{i,j}$. In particular, the number of periodic
    points of period $m$ of $(X_{G_A}, \sigma)$ is equal to 
    $\Trace (A^m)$.
\end{prop}


We include the computation of the zeta function for the Fibonacci shift. 
Compare with the Bowen-Lanford formula \cite{Bowen_Lanford_Zeta_function_of_subshift}.

\begin{figure}
\centering
    \begin{tikzpicture}  
    \node[state] (a) at (0,0) {$A$}; 
    \node[state] (b) [right =of a] {$B$};  
    \path[->] (a) edge[bend left=60, "2"] (b); 
    \path[->] (b) edge[bend left=60, "1"] (a);
    \draw[->] (a) to [out=270,in=180,looseness=5, "2"] (a);
      \end{tikzpicture}  
      \caption{Edge labeled graph for the Fibonacci SFT.}
      \label{fig:Fibonnaci_graph}
\end{figure}
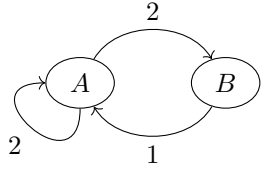

\begin{lemm}
    \label{lemm:A-M_zeta_function_of_the_SFT}
    For the left shift $ \sigma \colon \SFib \longrightarrow \SFib $, we have
    \[ \zeta_{\sigma} (t) = \frac{1}{1-t-t^2}. \]
\end{lemm}

\begin{proof}
    For the graph $G$ in Figure \ref{fig:Fibonnaci_graph} we have 
    $X_G = \SFib$. Then, by Proposition \ref{prop:periodic_pts_SFT}, 
    $N_n(\sigma|_{\SFib}) = \Trace (A_G^n)$.
    From Figure \ref{fig:Fibonnaci_graph}, we see that  
    \[ A_G \= 
    \begin{pmatrix}
        0 & 1 \\
        1 & 1
    \end{pmatrix}. \]
    Then we can write
    \[ \SFib = \{ \bm{w} = (w_i)_{\geq 0} \in \{1,2\}^{\N_0} \colon 
    A_{w_iw_{i+1}} = 1, \forall i \in \N_0 \}. \]
    With this description of $\SFib$, it can be proved inductively that for 
    $n \geq 1$
    \[ A^n \= 
    \begin{pmatrix}
        \ell_{n-1} & \ell_{n} \\
        \ell_{n} & \ell_{n+1}
    \end{pmatrix}, \]
    where $ \{ \ell_k \}_{k\geq 0}$ is the Fibonacci sequences. Thus, 
    $\ell_0 = 0$, $\ell_1 = 1$, and for $k \geq 2$, $\ell_k = \ell_{k-1} + 
    \ell_{k-2}$.
     Then $N_n(\sigma) = \ell_{n-1} + \ell_{n+1}$, and 
    \[ \zeta_{\sigma}(t) = \exp \left( \sum_{n\geq 1} (\ell_{n-1} + \ell_{n+1}) \frac{1}{n}t^n \right). \] So we compute the logarithmic 
    derivative

    \begin{align*}
        \frac{\zeta'_{\sigma}(t)}{\zeta_{\sigma}(t)} &= 
        \sum_{n \geq 1} (\ell_{n-1}+\ell_{n+1})t^{n-1}\\
        &= - \frac{1}{t} + (1 + \frac{1}{t^2})\sum_{n \geq 1} \ell_n t^n \\
        &= -\frac{1}{t} + (1 + \frac{1}{t^2})\frac{t}{1-t-t^2}.
    \end{align*}
    To get the last equality, we use the rational representation of the 
    generation function of the Fibonacci sequence.
    From the above we get \[ \frac{\zeta_{\sigma}'(t)}{\zeta_{\sigma}(t)} = 
    \frac{2t+1}{1-t-t^2},\] and finally
    \[ \ln (\zeta_{\sigma}(t)) = - \ln(1-t-t^2). \]
    This last equality gives us the desired formula.
\end{proof}


\section{Virtually unimodal maps}
\label{section_virtually_unimodal_maps} 

In this section, we introduce the notion of ``virtually unimodal maps", 
study their Artin-Mazur zeta function, and characterize their combinatorics 
in the case they have finite turning point orbits.

\begin{defi}
    \label{def:VUM}
    A multimodal map $f$ is called \emph{virtually unimodal} (VU) if there is
    a turning point $c$ of $f$ such that the following holds:
    \vspace{0.2cm}
    \begin{enumerate}
        \item[(VU1)] 
        $\Trn (f) \cap \interior (\langle f^2(c), f(c) \rangle) = \{ c \}$,
        \vspace{0.3cm}
        \item[(VU2)] $\cO_f(c) \subset \langle f^2(c),f(c) \rangle$, and
        \vspace{0.3cm}
        \item[(VU3)] there exists a $k \geq 0$ such that  $f^k(\Trn (f)) 
        \subset \langle f^2(c) , f(c) \rangle$.
    \end{enumerate}
    We call $c \in \Trn (f)$ the \emph{dominant turning point} of $f$.
\end{defi}

Conditions (VU1) and (VU2) imply that $f|_{\langle f^2(c),f(c) \rangle}$ is
unimodal with either $f^2(c) < c < f(c)$ or $f(c) < c < f^2(c)$. This 
implies that
\[f(\langle f^2(c),f(c) \rangle) = \langle f^2(c),f(c) \rangle.\]
Conditions (VU1) and
(VU3) do not rule out the possibility that  $f^2(c)$ or $f(c)$ are turning
points of $f$.

\begin{theo}
    \label{theo:VUM_A-M_zeta_function}
    Let $I \subset \R$ be a compact interval, and 
    $f \colon I \longrightarrow I$ a virtually unimodal map with dominant 
    turning point $c \in I$. The Artin-Mazur zeta function of $f$ is a rational
    function if and only if $c$ is asymptotic to a periodic orbit of $f$. 
\end{theo}

\begin{proof}

    Let $f \colon I \longrightarrow I$ be a VU map with turning points set
$\Trn (f) = \{ c_1, c_2, \ldots, c_m \}$, and dominant turning point 
$c_j$. Conditions (VU1) and (VU2) in Definition \ref{def:VUM} implies that
\[ \langle f^2(c) , f(c) \rangle \subset I_{j-1} \cup I_j. \] Then, Condition 
(VU3) in Definition \ref{def:VUM} implies that for every $c_i \in \Trn (f)$ 
with $i \neq j$, there are polynomials $p_{c_i,k}(t) \in \Z [t] $ with 
$k \in \{ 0,1, \ldots, m \} \setminus \{ j-1, j \}$ such that 
\begin{equation}
    \label{eq:VU_kneading_increment}
    \nu_i \= N_{i,j-1}I_{j-1} + N_{i,j}I_J + \sum_{k \in \{ 0,1, \ldots, m \} \setminus \{ j-1, j \} } p_{c_i,k}(t)I_{k}.
\end{equation}
Thus, only finitely many coefficients of the kneading increment $\nu_i$ 
contains symbols in $\{ I_0, I_1 \ldots I_m \} \setminus \{ I_{j-1}, I_j \}$,
and the kneading matrix takes the following form
\[
[N_{i,j}] = 
\begin{bmatrix}
p_{c_1,0} & p_{c_1,1} & \ldots & N_{c_1,j-1} & N_{c_1,j} & \ldots & p_{c_1,m}\\
p_{c_2,0} & p_{c_2,1} & \ldots & N_{c_2,j-1} & N_{c_2,j} & \ldots & p_{c_2,m}\\
\vdots & \vdots & & \vdots & \vdots & & \vdots  \\
0 & 0 & \ldots & N_{c_j,j-1} & N_{c_j,j} & \ldots & 0\\
\vdots & \vdots & & \vdots & \vdots & & \vdots  \\
p_{c_m,0} & p_{c_m,1} & \ldots & N_{c_m,j-1} & N_{c_m,j} & \ldots & p_{c_m,m}\\
\end{bmatrix}
\]
Then, deleting the $j$th column, we get 
\[
D_j = 
\begin{bmatrix}
p_{c_1,0} & p_{c_1,1} & \ldots & p_{c_1,j-2} & N_{c_1,j} & \ldots & p_{c_1,m}\\
p_{c_2,0} & p_{c_2,1} & \ldots & p_{c_2,j-2} & N_{c_2,j} & \ldots & p_{c_2,m}\\
\vdots & \vdots & & \vdots & \vdots & & \vdots  \\
0 & 0 & \ldots & 0 & N_{c_j,j} & \ldots & 0\\
\vdots & \vdots & & \vdots & \vdots & & \vdots  \\
p_{c_m,0} & p_{c_m,1} & \ldots & p_{c_m,j-2} & N_{c_m,j} & \ldots & p_{c_m,m}\\
\end{bmatrix}
\]
and \[ \det (D_j) = P(t)N_{c_j,j}, \] where $P(t)$ is a polynomial.
This implies that $D_f(t)$ is a rational function if and only if $N_{c_j,j}$
is a rational function. 

Finally, since 
\[ \cO_f(c_j) \subseteq \langle f^2(c_j), f(c_j) \rangle \subseteq 
I_{j-1} \cup I_j, \]  
and $f|_{\langle f^2(c_j), f(c_j) \rangle}$ is unimodal, using 
\eqref{eq:kneading_determinant_formula}, we have that 
\[D_{f|_{\langle f^2(c_j), f(c_j) \rangle}}(t) = 
\frac{N_{c_j,j}}{1- \e(I_{j-1})}.\]
So, $N_{c_j,j}$ is a rational function if and only if 
$D_{f|_{\langle f^2(c_j), f(c_j) \rangle}}(t)$ is a rational function,
and by Theorem \ref{theo:Theorem_1}, 
$D_{f|_{\langle f^2(c_j), f(c_j) \rangle}}(t)$ is a rational function if and 
only if $c_j$, the leading turning point of $f$, is asymptotic to a 
periodic cycle.

\end{proof}

\subsection{Combinatorics of virtually unimodal maps} \label{subsection_Combinatorics_of_VU_maps}

Of particular interest is the case of VU maps with finite turning point
orbits as they will correspond to post-critically finite polynomials. 
We will follow the approach in \cite{Bonifan-Milnor_Surtherland2021_Thuston_alg_real_poly} to describe the combinatorics 
of these maps.
We will say that a piecewise monotone map 
$f \colon [a,b] \longrightarrow [a,b]$ has \emph{finite turning point orbits}
if the set $\Trn (f)$ is nonempty, and the set
\[ \cO_{f}(\Trn (f)) \= \bigcup_{c \in \Trn (f)} \cO_f(c), \] 
is finite. 
Thus, every turning point of $f$ is either periodic or preperiodic.
In this case, we can write
\[ \{a,b\}\cup \bigcup_{k \geq 0} f^{k}(\Trn (f)) = \{x_0, x_1,\ldots, x_n\},\]
with \[a = x_0 < x_1 < x_2 < \ldots < x_n = b.\] Then, for every 
$i \in \{ 0,1, \ldots , n \}$ there exists a $\rho_i \in \{ 0,1, \ldots , n \}$
such that $f(x_i) = x_{\rho_i}$. We call the vector 
\[ \ovra{\varrho}(f) = (\rho_0,\rho_1, \ldots, \rho_n), \] the 
\emph{combinatorics of $f$}.
The combinatorics of a map with finite turning point orbits represents the
action of $f$ restricted to $\{ a, b \} \cup \cO_{f}(\Trn (f))$. 

Our first goal is to characterize the combinatorics  of VU maps. Thus, we want
to know when a given vector $\ovra{\rho} \in \{ 0,1, \ldots , n \}^{n+1}$ 
represents the combinatorics of a VU map. 
For this purpose, it will be convenient to give a geometric representation of 
a given combinatorics. A \emph{piecewise linear model} for a vector
\[\ovra{\rho} = (\rho_0, \ldots, \rho_n) \in \{ 0,1, \ldots, n \}^{n+1},\] 
is a map $F_{\ovra{\rho}} \colon [0,n] \longrightarrow [0,n]$ which maps each
integer $j \in [0,n]$ to $\rho_j$ and is linear between integers. 
Thus, for each $j \in \{ 0,1, \ldots, n \}$ we have that 
\begin{equation}
    \label{eq:PM_model_def}
    (F_{\ovra{\rho}}|_{[j,j+1]})(x) = (\rho_{j+1}-\rho_j)(x -j) + \rho_j.
\end{equation}

With this geometric model we can identify turning points, orbits, and periodic 
cycles of a given vector $\ovra{\rho} \in \{ 0,1, \ldots , n \}^{n+1}$.
We say that $\ovra{\rho}$ has a \emph{turning point} in the $i$th position if
$i$ is a turning point for $F_{\ovra{\rho}}$. Thus, 
\[ \rho_{i-1} < \rho_i \hspace{1cm} \text{ and} \hspace{1cm} 
\rho_{i+1} < \rho_i, \] or
\[ \rho_{i-1} > \rho_i \hspace{1cm} \text{ and} \hspace{1cm} 
\rho_{i+1} > \rho_i. \] We denote the set of turning points of $\ovra{\rho}$
by $\Trn (\ovra{\rho})$.
For $i \in \{0, 1, \ldots, n\}$ we define the $\ovra{\rho}$-\emph{orbit} of 
$i$ as the orbit under $F_{\ovra{\rho}}$ of $i$.
We will use the notation $\cO_{F_{\ovra{\rho}}}(i)$ to denote the 
$\ovra{\rho}$-orbit of $i$.

We say that $\ovra{\rho}$ has a \emph{periodic point} in the $i$th position
if $i$ is a periodic point of $F_{\ovra{\rho}}$. In this case, we call the
vector
\[ (F_{\ovra{\rho}}(i), F_{\ovra{\rho}}^2(i), \ldots , 
F_{\ovra{\rho}}^{p(i) -1}(i), i), \]
a \emph{periodic cycle} of $\ovra{\rho}$, where 
\[ p(i) \= \min \{ p \geq 1 \colon F_{\ovra{\rho}}^p(i) = i \}. \] A periodic
cycle is \emph{post-turning} if it is contained in the 
$\ovra{\rho}-$orbit of a turning point. The \emph{length} of a cycle is equal
to the number of distinct elements in it.

Additionally, we will say that $\ovra{\rho}$ is \emph{framed} if 
$\rho_0, \rho_n$ belongs to $\{0,n\}$. Thus, $\ovra{\rho}$ is framed if
$F_{\ovra{\rho}}$ is boundary anchored ($F_{\ovra{\rho}}(\{0,n\}) \subset \{0,n\}$).
\begin{rema}
    It is easy to see that given a vector 
    $\ovra{\rho} \in \{0,1, \ldots, n \}^n$ we can extend $F_{\ovra{\rho}}$
    to $[-1,n+1]$ so that $F_{\ovra{\rho}}(\{-1,n+1\}) \subseteq \{-1,n+1\}$
    and the order of the orbits of the turning points is the same.
\end{rema}

Observe that for a framed combinatorics $\ovra{\rho}$, the entries $\{0,n\}$
form a cycle of period two, or they are fixed. From now on, all the combinatorics we consider will be framed unless explicitly stated.



\begin{lemm}
    \label{lemm:PM_combinatorics}
    Let $\ovra{\rho} \in \{ 0,1, \ldots , n \}^{n+1}$. The piecewise linear 
    map $F_{\ovra{\rho}}$ is PM if and only if
    \begin{enumerate}
        \item $\rho_i \neq \rho_{i+1}$ for all $0 \leq i < n$.
    \end{enumerate}
    Additionally, $\ovra{\varrho}(F_{\ovra{\rho}})$ coincides with 
    $\ovra{\rho}$ if and only if
    \begin{enumerate}
        \item[(2)] every cycle of $\ovra{\rho}$, except possibly the boundary
        cycle, is post-turning. Thus
        \[ \{1,2, \ldots, n-1\} \subset \bigcup_{c \in \Trn(\ovra{\rho})} \cO_{F_{\ovra{\rho}}}(c). \]
    \end{enumerate}
\end{lemm}

\begin{proof}
    Let 
    $\ovra{\rho} = (\rho_0, \dots , \rho_n) \in \{ 0, 1, \ldots, n \}^{n+1}$.
    First, observe that by definition 
    $\Trn (F_{\ovra{\rho}}) \subset \{ 1, \ldots, n-1 \}$, and by 
    \eqref{eq:PM_model_def}, $F_{\ovra{\rho}}$ is PM if and only if for every
    $0 \leq i < n$ we have 
    \[(F_{\ovra{\rho}}|_{(i,i+1)})'(x) = \rho_{i+1} - \rho_i \neq 0.\] This proves the first part of the lemma.

    Now, if $\ovra{\varrho}(F_{\ovra{\rho}}) = \ovra{\rho}$, then 
    \[ \cO_{F_{\ovra{\rho}}}(\Trn (F_{\ovra{\rho}})) = \{ 0, n \} \cup
    \bigcup_{c \in \Trn (F_{\ovra{\rho}}) } \cO_{F_{\ovra{\rho}}}(c) = 
    \{ 0,1, \ldots , n\}.\] So for every $ i \in \{ 1, \ldots , n-1 \}$ there
    is a $c \in \Trn (F_{\ovra{\rho}})$ such that 
    $i \in \cO_{F_{\ovra{\rho}}}(c)$. Thus, there exists a $m \geq 0$ so that 
    $i = F_{\ovra{\rho}}^m (c)$. If $m = 0$, then 
    $i \in \Trn (F_{\ovra{\rho}})$ and we are done. If $m > 0$, put 
    $k_{c,j} = F_{\ovra{\rho}}^j(c)$ for every $1 \leq j \leq m$. Then
    $k_{c,1} = F_{\ovra{\rho}}(c) = \rho_c$, for $1 \leq j \leq m $
    \[ k_{c,j+1} = F_{\ovra{\rho}}^{j+1}(c) = 
    F_{\ovra{\rho}}(F_{\ovra{\rho}}^{j}(c)) = 
    F_{\ovra{\rho}}(k_{c,j}) = \rho_{k_{c,j}}, \] and 
    \[ \rho_{k_{c,m}} = F_{\ovra{\rho}}(k_{c,m-1}) = 
    F_{\ovra{\rho}}(F_{\ovra{\rho}}^{m-1}(c)) = F_{\ovra{\rho}}^m(c) = i.\]
    So point (2) holds.

    Suppose that (2) holds. Let $i \in \{ 1, \ldots , n-1 \}, $ and let 
    $c, k_{c,1}, \ldots , k_{c,m} \in \{ 0, \ldots , n \}$ be as in (2). Then 
    \begin{align*}
        F_{\ovra{\rho}}(c) &= \rho_c = k_{c,1}, \\
        F_{\ovra{\rho}}^2(c) &= F_{\ovra{\rho}}(F_{\ovra{\rho}}(c)) = 
        F_{\ovra{\rho}}(k_{c,1}) = \rho_{k_{c,1}} = k_{c,2}, \\
        F_{\ovra{\rho}}^3(c) &= F_{\ovra{\rho}}(F_{\ovra{\rho}}^2(c)) = 
        F_{\ovra{\rho}}(k_{c,2}) = \rho_{k_{c,2}} = k_{c,3} \\
        \vdots \\
        F_{\ovra{\rho}}^{m+1}(c) &= F_{\ovra{\rho}}(F_{\ovra{\rho}}^m(c)) = 
        F_{\ovra{\rho}}(k_{c,m}) = \rho_{k_{c,m}} = i.
    \end{align*}
    Thus, $i \in \cO_{F_{\ovra{\rho}}}(\Trn (F_{\ovra{\rho}}))$. Then
    \[\{1, \ldots , n-1 \} \subseteq \cO_{F_{\ovra{\rho}}}(\Trn (F_{\ovra{\rho}})).\]
    This implies that $\ovra{\varrho}(F_{\ovra{\rho}}) = \ovra{\rho}$.
\end{proof}


The following examples show the situations presented when conditions (1) 
and (2) in the previous lemma fail. 

The vector $\ovra{\rho} = (0,3,3,2,0)$ does not satisfy condition (1) because 
$\rho_1 = \rho_2 = 3$. Then, it is not realized as the combinatorics of any PM
map $f$ since, if it does, we
will have two consecutive points $ x_1,x_2$ with 
$f(x_1) = f(x_2) = x_3$. This implies that $f|_{[x_1,x_2]}$ is constant or 
it has a turning point in $(x_1,x_2)$ but there is no extra marked point in 
$(x_1,x_2)$, see Figure \ref{figu:Pl_bad_examples}.

The vector $\ovra{\rho} = (0,3,4,7,6,5,2,1,0)$ does not satisfy condition 
(2). Indeed, $\ovra{\rho}$ has three cycles in $\{1,2,3,4,5,6,7\}$, these are 
$(1,3,7)$, $(2,4,6)$, and $(5)$, and a single turning point at $i = 3$. Since
\[\cO_{F_{\ovra{\rho}}}(3) = \{1,3,7\},\] we have that 
\[\ovra{\varrho}(F_{\ovra{\rho}}) = (0,1,3,7,0) \neq \ovra{\rho}.\]


\begin{figure}[h!]
  \raggedright
  \begin{subfigure}[b]{.3 \textwidth}
  \centering
    \begin{tikzpicture}[line cap=round,line join=round,>=triangle 45,x=1.1cm,y=1.1cm]
    \clip(-0.3,-0.3) rectangle (4.3,4.3);
      \draw[line width=0.8, gray!25] (4,0)-- (0,0);
      \draw[line width=0.8, gray!25] (0,0)-- (0,4);
      \draw[line width=0.8, gray!25] (0,4)-- (4,4);
      \draw[line width=0.8, gray!25] (4,4)-- (4,0);
      \draw[line width=0.8, gray!25] (0,1)-- (4,1);
      \draw[line width=0.8, gray!25] (0,2)-- (4,2);
      \draw[line width=0.8, gray!25] (0,3)-- (4,3);
      \draw[line width=0.8, gray!25] (1,0)-- (1,4); 
      \draw[line width=0.8, gray!25] (2,0)-- (2,4);
      \draw[line width=0.8, gray!25] (3,0)-- (3,4);
      \draw[line width=0.9] (0,0) -- (1,3);
      \draw[line width=0.9] (1,3) -- (2,3);
      \draw[line width=0.9] (2,3) -- (3,2);
      \draw[line width=0.9] (3,2) -- (4,0);
      \draw [dash pattern=on 1.5pt off 1.5pt, gray] (0,0)-- (4,4);
    \end{tikzpicture}
   \end{subfigure}
   \hspace{3cm}
   \begin{subfigure}[b]{.3 \textwidth}
   \centering
   \begin{tikzpicture}[line cap=round,line join=round,>=triangle 45,x=1.1cm,y=1.1cm]
    \clip(-0.3,-0.3) rectangle (4.3,4.3);
      \draw[line width=0.8, gray!25] (4,0)-- (0,0);
      \draw[line width=0.8, gray!25] (0,0)-- (0,4);
      \draw[line width=0.8, gray!25] (0,4)-- (4,4);
      \draw[line width=0.8, gray!25] (4,4)-- (4,0);
      \draw[line width=0.8, gray!25] (0,0.5)-- (4,0.5);
      \draw[line width=0.8, gray!25] (0,1)-- (4,1);
      \draw[line width=0.8, gray!25] (0,1.5)-- (4,1.5);
      \draw[line width=0.8, gray!25] (0,2)-- (4,2);
      \draw[line width=0.8, gray!25] (0,2.5)-- (4,2.5);
      \draw[line width=0.8, gray!25] (0,3)-- (4,3);
      \draw[line width=0.8, gray!25] (0,3.5)-- (4,3.5);
      \draw[line width=0.8, gray!25] (0.5,0)-- (0.5,4);
      \draw[line width=0.8, gray!25] (1,0)-- (1,4);
      \draw[line width=0.8, gray!25] (1.5,0)-- (1.5,4);
      \draw[line width=0.8, gray!25] (2,0)-- (2,4);
      \draw[line width=0.8, gray!25] (2.5,0)-- (2.5,4);
      \draw[line width=0.8, gray!25] (3,0)-- (3,4);
      \draw[line width=0.8, gray!25] (3.5,0)-- (3.5,4);
      \draw[line width=0.9] (0,0) -- (0.5,1.5);
      \draw[line width=0.9] (0.5,1.5) -- (1,2);
      \draw[line width=0.9] (1,2) -- (1.5,3.5);
      \draw[line width=0.9] (1.5,3.5) -- (2,3);
      \draw[line width=0.9] (2,3) -- (2.5,2.5);
      \draw[line width=0.9] (2.5,2.5) -- (3,1);
      \draw[line width=0.9] (3,1) -- (3.5,0.5);
      \draw[line width=0.9] (3.5,0.5) -- (4,0);
      \draw [dash pattern=on 1.5pt off 1.5pt, gray] (0,0)-- (4,4);
    \end{tikzpicture}
   \end{subfigure}
  \caption{The graphics of the piecewise linear models associated to the 
  combinatorics $\ovra{\rho} = (0,3,3,2,0)$ (left figure) and 
  $\ovra{\rho} = (0,3,4,7,6,5,2,1,0)$ (right figure).}
  \label{figu:Pl_bad_examples}
\end{figure}
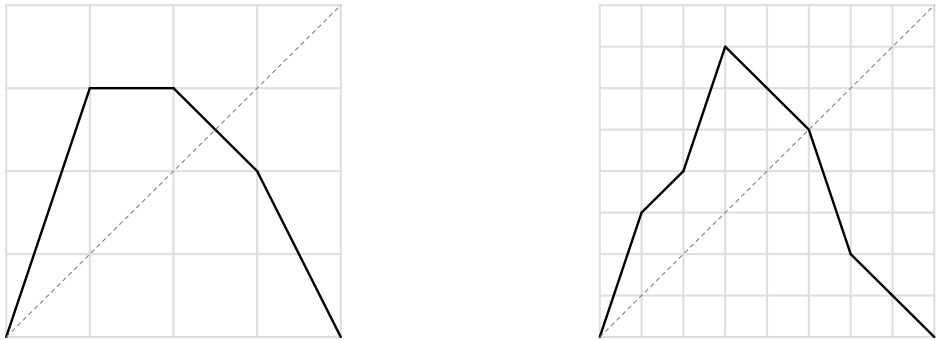





We will say that a vector $\ovra{\rho} \in \{ 0,1, \ldots, n\}^{n+1}$ is 
a \emph{virtually unimodal combinatorics} if it satisfies condition (1) and 
(2) in Lemma \ref{lemm:PM_combinatorics} and the piecewise linear map
$F_{\ovra{\rho}}$ is virtually unimodal. 

As an example, the vector $(5,2,3,4,2,0)$ is a virtually unimodal 
combinatorics, with $\Trn (\ovra{\rho}) = \{ 1,3\}$ and dominant turning 
point $3$, see Figure \ref{figu:VU_combinatorics_examples}.

On the other hand, the vector $(6,2,1,4,5,3,0)$ is not virtually unimodal.
It has two turning points $\Trn (\ovra{\rho}) = \{1,4\}$, and condition 
(VUC3) does not hold since 
\[ \cO_{F_{\ovra{\rho}}}(1) = \{1,2\} \hspace{1cm} \text{ and } \hspace{1cm}
\cO_{F_{\ovra{\rho}}}(4) = \{4,5,3\}.\] So, we have that $k_{4,1} = 5$ and 
$k_{4,2} = 3$, then 
\[1 < \min (k_{4,1}, k_{k,2}) = 3 \hspace{1cm} \text{ and } \hspace{1cm} 
2 < \min (k_{4,1}, k_{k,2}) = 3.\] 


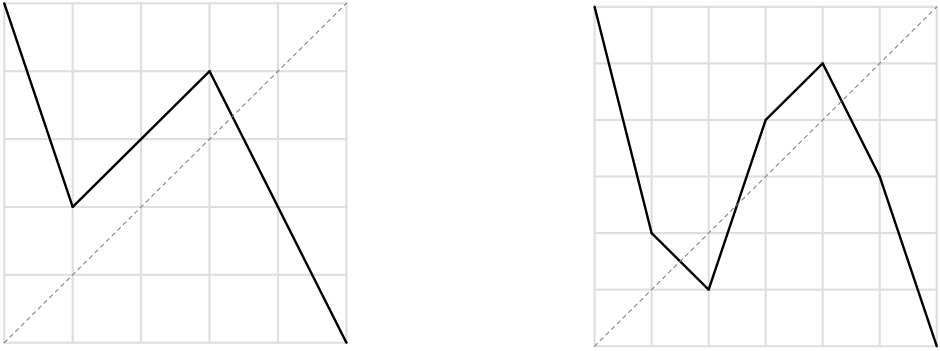
\begin{figure}[h!]
  \raggedright
  \begin{subfigure}[b]{.3 \textwidth}
  \centering
    \begin{tikzpicture}[line cap=round,line join=round,>=triangle 45,x=0.9cm,y=0.9cm]
    \clip(-0.3,-0.3) rectangle (5.3,5.3);
      \draw[line width=0.8, gray!25] (5,0)-- (0,0);
      \draw[line width=0.8, gray!25] (0,0)-- (0,5);
      \draw[line width=0.8, gray!25] (0,5)-- (5,5);
      \draw[line width=0.8, gray!25] (5,5)-- (5,0);
      \draw[line width=0.8, gray!25] (0,1)-- (5,1);
      \draw[line width=0.8, gray!25] (0,2)-- (5,2);
      \draw[line width=0.8, gray!25] (0,3)-- (5,3);
      \draw[line width=0.8, gray!25] (1,0)-- (1,5); 
      \draw[line width=0.8, gray!25] (2,0)-- (2,5);
      \draw[line width=0.8, gray!25] (3,0)-- (3,5);
      \draw[line width=0.8, gray!25] (4,0)-- (4,5);
      \draw[line width=0.8, gray!25] (0,4)-- (5,4);
      \draw[line width=0.9] (0,5) -- (1,2);
      \draw[line width=0.9] (1,2) -- (2,3);
      \draw[line width=0.9] (2,3) -- (3,4);
      \draw[line width=0.9] (3,4) -- (4,2);
      \draw[line width=0.9] (4,2) -- (5,0);
      \draw [dash pattern=on 1.5pt off 1.5pt, gray] (0,0)-- (5,5);
    \end{tikzpicture}
   \end{subfigure}
   \hspace{3cm}
   \begin{subfigure}[b]{.3 \textwidth}
   \centering
   \begin{tikzpicture}[line cap=round,line join=round,>=triangle 45,x=0.75cm,y=0.75cm]
    \clip(-0.3,-0.3) rectangle (7.3,7.3);
      \draw[line width=0.8, gray!25] (6,0)-- (0,0);
      \draw[line width=0.8, gray!25] (0,0)-- (0,6);
      \draw[line width=0.8, gray!25] (0,6)-- (6,6);
      \draw[line width=0.8, gray!25] (6,6)-- (6,0);
      \draw[line width=0.8, gray!25] (0,1)-- (6,1);
      \draw[line width=0.8, gray!25] (0,2)-- (6,2);
      \draw[line width=0.8, gray!25] (0,3)-- (6,3);
      \draw[line width=0.8, gray!25] (0,4)-- (6,4);
      \draw[line width=0.8, gray!25] (0,5)-- (6,5);
      \draw[line width=0.8, gray!25] (1,0)-- (1,6);
      \draw[line width=0.8, gray!25] (2,0)-- (2,6);
      \draw[line width=0.8, gray!25] (3,0)-- (3,6);
      \draw[line width=0.8, gray!25] (4,0)-- (4,6);
      \draw[line width=0.8, gray!25] (5,0)-- (5,6);
      \draw[line width=0.9] (0,6) -- (1,2);
      \draw[line width=0.9] (1,2) -- (2,1);
      \draw[line width=0.9] (2,1) -- (3,4);
      \draw[line width=0.9] (3,4) -- (4,5);
      \draw[line width=0.9] (4,5) -- (5,3);
      \draw[line width=0.9] (5,3) -- (6,0);
      \draw [dash pattern=on 1.5pt off 1.5pt, gray] (0,0)-- (6,6);
    \end{tikzpicture}
   \end{subfigure}
  \caption{On the left, the graphic of the piecewise linear model associated to the virtually unimodal combinatorics $(5,2,3,4,2,0)$. On the right, the 
  graphic of the piecewise linear model associated to the combinatorics
  $(6,2,1,4,5,3,0)$ (not virtually unimodal).}
  \label{figu:VU_combinatorics_examples}
\end{figure}

\subsection{Constructing virtually unimodal combinatorics} \label{subsection_construction_VU_combinatorics}

We present a method to construct virtually unimodal combinatorics with any 
prescribed number of turning points, starting from a unimodal combinatorics.

This method is not general, though it is enough for our purpose, and it can be
extended without major changes.

First, we start with the unimodal combinatorics $\ovra{\rho}_0=(0,2,3,1,0)$.
This combinatorics has a single turning point at $2$ that is periodic of 
period three. So the piecewise linear map $F_{\ovra{\rho}_0}$ is unimodal
and its turning point is periodic of period three, see Figure \ref{figure:period_three_comb}.
By Sarkovskii's Theorem, $F_{\ovra{\rho}_0}$ has periodic points of every 
period. Moreover, periodic points are dense in $[1,3]$.

\begin{figure}[h!]
  \centering
    \begin{tikzpicture}[line cap=round,line join=round,>=triangle 45,x=1.1cm,y=1.1cm]
    \clip(-0.3,-0.3) rectangle (4.3,4.3);
      \draw[line width=0.8, gray!25] (4,0)-- (0,0);
      \draw[line width=0.8, gray!25] (0,0)-- (0,4);
      \draw[line width=0.8, gray!25] (0,4)-- (4,4);
      \draw[line width=0.8, gray!25] (4,4)-- (4,0);
      \draw[line width=0.8, gray!25] (0,1)-- (4,1);
      \draw[line width=0.8, gray!25] (0,2)-- (4,2);
      \draw[line width=0.8, gray!25] (0,3)-- (4,3);
      \draw[line width=0.8, gray!25] (1,0)-- (1,4); 
      \draw[line width=0.8, gray!25] (2,0)-- (2,4);
      \draw[line width=0.8, gray!25] (3,0)-- (3,4);
      \draw[line width=0.9] (0,0) -- (1,2);
      \draw[line width=0.9] (1,2) -- (2,3);
      \draw[line width=0.9] (2,3) -- (3,1);
      \draw[line width=0.9] (3,1) -- (4,0);
      \draw[-to, color=red] (2,2) -- (2,2.5);
      \draw[ color=red] (2,2.5) -- (2,3);
      \draw[-to, color=red] (2,3) -- (2.5,3);
      \draw[ color=red] (2.5,3) -- (3,3);
      \draw[-to, color=red] (3,3) -- (3,2);
      \draw[ color=red] (3,2) -- (3,1);
      \draw[-to, color=red] (3,1) -- (2,1);
      \draw[ color=red] (2,1) -- (1,1);
      \draw[-to, color=red] (1,1) -- (1,1.5);
      \draw[ color=red] (1,1.5) -- (1,2);
      \draw[-to, color=red] (1,2) -- (1.5,2);
      \draw[ color=red] (1.5,2) -- (2,2);
      \draw [dash pattern=on 1.5pt off 1.5pt, gray] (0,0)-- (4,4);
    \end{tikzpicture}
  \caption{Piecewise linear model associated with the combinatoric $\ovra{\rho}_0 = (0,2,3,1,0)$.}
  \label{figure:period_three_comb}
\end{figure}
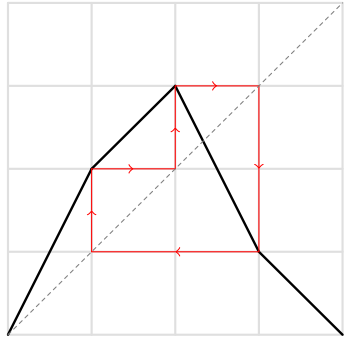    

To construct a virtually unimodal combinatoric with $\nu \geq 2$ turning
points we do the following:
\begin{itemize}
    \item[1st] Choose $\nu -1$ periodic points $k_{1,0}, k_{2,0}
    \ldots, k_{\nu-1,0}$ of $F_{\ovra{\rho}_0}$ in $[1,3]$ in such a way 
    that $k_{\nu-1,0} \leq 2$ and
    for every $1 \leq i \leq \nu -1$ we have 
    \[ \sg (k_{i,0} - k_{i+1,0}) = (-1)^i \text{ or } (-1)^{i+1}.\]
    \item[2nd] Mark the orbit of the periodic points and the orbit of the 
    turning point 
    \[\{ 1,2,3 \} \cup \bigcup_{i=1}^{\nu-1} \cO_{F_{\ovra{\rho}_0}}
    (k_{i,0}) = \{ y_1,y_2, \ldots , y_{n'} \} \] with 
    \[y_1 < y_2 < \ldots < y_{n'}.\]
    Mark the position of the chosen periodic points 
    $k_{1,0}, k_{2,0}, \ldots , k_{\nu-1,0} $ and the turning point $2$. 
    Thus, let $k_{1,1}, k_{2,1}, \ldots , k_{\nu -1, 1}, c \in \N$ be 
    so that $y_{k_{i,1}} = k_{i,0}$ for all $1 \leq i \leq \nu-1$ and 
    $y_c = 2$. Write the vector 
    \[ \ovra{\xi} = (\xi_1, \xi_2, \ldots , \xi_{n'}) \] so that 
    $F_{\ovra{\rho}_0}(y_i) = y_{\xi_i}$ for all $1 \leq i \leq n'$.
    \item[3rd] Write $\ovra{\xi}$ as a vector in 
    $\{ 0,1, \ldots , n'+\nu+1 \}^{n'+\nu+1}$
    in the following way. Put \[ \ovra{\rho}= (\rho_0, \rho_1, 
    \ldots, \rho_{n'+\nu +1} ) \] where
    \[ \rho_0 = 
    \begin{cases}
        0 & \text{ if } \nu -1 \equiv 0 (\mod{2})\\
        n' + \nu + 1 & \text{ if } \nu -1 \equiv 1 (\mod{2})
    \end{cases}\]
    \[\rho_i = 
    \begin{cases}
        k_{i,1}  & \text{ if } 1 \leq i \leq \nu -1 \\
        \xi_i  & \text{ if } \nu \leq i \leq \nu + n' -1 \\
        0 & \text{ if } i = \nu + n'
    \end{cases} \]
\end{itemize}

By construction, we have that the vector $\ovra{\rho}$  is a virtually 
unimodal combinatorics, with the dominant turning point being the unique 
turning point of $\ovra{\rho}$ in 
$\{ \rho_{\nu+1}, \rho_{\nu+2}, \ldots ,\rho_{\nu + n' +2}  \}$.


\begin{exam}
    \label{example_VU_combinatorics}
    Let $\nu \geq 2$. We start with the combinatorics 
    $\ovra{\rho}_0 = (0,2,3,1,0)$, and we choose the unique periodic point 
    of minimal period two of $F_{\ovra{\rho}_0}$ in $[1,2]$, $k_1$ and its 
    image $k_2 \= F_{\ovra{\rho}}(k_1) \in [2,3]$. If we mark them, we get 
    the set \[ \{ 1, k_1, 2, k_2,3 \}. \] Now, we construct the combinatorics
    associated with these marked points
    \[ \ovra{\xi} = (3,4,5,2,1). \]
    Finally, we write the combinatorics $\ovra{\rho}_{F}(\nu)$ including the
    $\nu -1$
    turning points whose images will oscillate between $k_1$ and $k_2$.
    \begin{itemize}
        \item If $\nu \geq 2$ is even, we write
        \[ \ovra{\rho}_F(\nu) = (\nu + 5, \nu + 1, \nu + 3, \ldots, 
        \nu + 1, \rho_{\nu } = \nu + 2, \nu +  3, \nu + 4, \nu + 1,
        \nu, 0); \]
        \item If $\nu > 2$ is odd, we write
        \[ \ovra{\rho}_F(\nu) = (0, \nu + 3, \nu + 1, \ldots, \nu + 1, 
        \rho_{\nu} = \nu + 2,
        \nu + 3, \nu + 4, \nu + 1, \nu, 0). \]
    \end{itemize}
    We are marking the coordinate $\nu+1$ in the vector where $\ovra{\xi}$
    start.

    For $ \nu = 2$, we get the combinatorics 
    \[ \ovra{\rho}_F(2) = (7,3,4,5,6,3,2,0), \] and for $\nu = 3$ we get the 
    combinatorics \[ \ovra{\rho}_F(3) = (0,6,4,5,6,7,4,3,0). \] See 
    Figure \ref{figu:VU_comb_examples_for_l3_and_l4} for the 
    graph of the piecewise linear models associated with each of these 
    combinatorics.
\end{exam}


\begin{figure}[h!]
  \raggedright
  \begin{subfigure}[b]{.3 \textwidth}
  \centering
    \begin{tikzpicture}[line cap=round,line join=round,>=triangle 45,x=0.85cm,y=0.85cm]
    \clip(-0.3,-0.3) rectangle (7.3,7.3);
      \draw[line width=0.8, gray!25] (7,0)-- (0,0);
      \draw[line width=0.8, gray!25] (0,0)-- (0,7);
      \draw[line width=0.8, gray!25] (0,7)-- (7,7);
      \draw[line width=0.8, gray!25] (7,7)-- (7,0);
      \draw[line width=0.8, gray!25] (0,1)-- (7,1);
      \draw[line width=0.8, gray!25] (0,2)-- (7,2);
      \draw[line width=0.8, gray!25] (0,3)-- (7,3);
      \draw[line width=0.8, gray!25] (1,0)-- (1,7); 
      \draw[line width=0.8, gray!25] (2,0)-- (2,7);
      \draw[line width=0.8, gray!25] (3,0)-- (3,7);
      \draw[line width=0.8, gray!25] (4,0)-- (4,7);
      \draw[line width=0.8, gray!25] (0,4)-- (7,4);
      \draw[line width=0.8, gray!25] (5,0)-- (5,7);
      \draw[line width=0.8, gray!25] (0,5)-- (7,5);
      \draw[line width=0.8, gray!25] (6,0)-- (6,7);
      \draw[line width=0.8, gray!25] (0,6)-- (7,6);
      \draw[line width=0.9] (0,7) -- (1,3);
      \draw[line width=0.9] (1,3) -- (2,4);
      \draw[line width=0.9] (2,4) -- (3,5);
      \draw[line width=0.9] (3,5) -- (4,6);
      \draw[line width=0.9] (4,6) -- (5,3);
      \draw[line width=0.9] (5,3) -- (6,2);
      \draw[line width=0.9] (6,2) -- (7,0);
      \draw[-to, color=red] (2,4) -- (3,4);
      \draw[ color=red] (3,4) -- (4,4);
      \draw[-to, color=red] (4,4) -- (4,5);
      \draw[ color=red] (4,5) -- (4,6);
      \draw[-to, color=red] (4,6) -- (5,6);
      \draw[ color=red] (5,6) -- (6,6);
      \draw[-to, color=red] (6,6) -- (6,4);
      \draw[ color=red] (6,4) -- (6,2);
      \draw[-to, color=red] (6,2) -- (4,2);
      \draw[ color=red] (4,2) -- (2,2);
      \draw[-to, color=red] (2,2) -- (2,3);
      \draw[ color=red] (2,3) -- (2,4);
      \draw[-to, color=blue] (3,3) -- (3,3.75);
      \draw[ color=blue] (3,3.75) -- (3,5);
      \draw[-to, color=blue] (3,5) -- (3.75,5);
      \draw[ color=blue] (3.75,5) -- (5,5);
      \draw[-to, color=blue] (5,5) -- (5,4);
      \draw[ color=blue] (5,4) -- (5,3);
      \draw[-to, color=blue] (5,3) -- (4,3);
      \draw[ color=blue] (4,3) -- (3,3);
      \draw[-to, color=blue] (1,3) -- (2.5,3);
      \draw[ color=blue] (2.5,3) -- (3,3);
      \draw [dash pattern=on 1.5pt off 1.5pt, gray] (0,0)-- (7,7);
    \end{tikzpicture}
   \end{subfigure}
   \hspace{3cm}
   \begin{subfigure}[b]{.3 \textwidth}
   \centering
   \begin{tikzpicture}[line cap=round,line join=round,>=triangle 45,x=0.75cm,y=0.75cm]
    \clip(-0.3,-0.3) rectangle (8.3,8.3);
      \draw[line width=0.8, gray!25] (8,0)-- (0,0);
      \draw[line width=0.8, gray!25] (0,0)-- (0,8);
      \draw[line width=0.8, gray!25] (0,8)-- (8,8);
      \draw[line width=0.8, gray!25] (8,8)-- (8,0);
      \draw[line width=0.8, gray!25] (0,1)-- (8,1);
      \draw[line width=0.8, gray!25] (0,2)-- (8,2);
      \draw[line width=0.8, gray!25] (0,3)-- (8,3);
      \draw[line width=0.8, gray!25] (0,4)-- (8,4);
      \draw[line width=0.8, gray!25] (0,5)-- (8,5);
      \draw[line width=0.8, gray!25] (0,6)-- (8,6);
      \draw[line width=0.8, gray!25] (0,7)-- (8,7);
      \draw[line width=0.8, gray!25] (1,0)-- (1,8);
      \draw[line width=0.8, gray!25] (2,0)-- (2,8);
      \draw[line width=0.8, gray!25] (3,0)-- (3,8);
      \draw[line width=0.8, gray!25] (4,0)-- (4,8);
      \draw[line width=0.8, gray!25] (5,0)-- (5,8);
      \draw[line width=0.8, gray!25] (6,0)-- (6,8);
      \draw[line width=0.8, gray!25] (7,0)-- (7,8);
      \draw[line width=0.9] (0,0) -- (1,6);
      \draw[line width=0.9] (1,6) -- (2,4);
      \draw[line width=0.9] (2,4) -- (3,5);
      \draw[line width=0.9] (3,5) -- (4,6);
      \draw[line width=0.9] (4,6) -- (5,7);
      \draw[line width=0.9] (5,7) -- (6,4);
      \draw[line width=0.9] (6,4) -- (7,3);
      \draw[line width=0.9] (7,3) -- (8,0);
      \draw[-to, color=red] (3,3) -- (3,4);
      \draw[ color=red] (3,4) -- (3,5);
      \draw[-to, color=red] (3,5) -- (4,5);
      \draw[ color=red] (4,5) -- (5,5);
      \draw[-to, color=red] (5,5) -- (5,6);
      \draw[ color=red] (5,6) -- (5,7);
      \draw[-to, color=red] (5,7) -- (6,7);
      \draw[ color=red] (6,7) -- (7,7);
      \draw[-to, color=red] (7,7) -- (7,5);
      \draw[ color=red] (7,5) -- (7,3);
      \draw[-to, color=red] (7,3) -- (5,3);
      \draw[ color=red] (5,3) -- (3,3);
      \draw[-to, color=blue] (4,4) -- (4,4.75);
      \draw[ color=blue] (4,4.75) -- (4,6);
      \draw[-to, color=blue] (4,6) -- (4.75,6);
      \draw[ color=blue] (4.75,6) -- (6,6);
      \draw[-to, color=blue] (6,6) -- (6,5);
      \draw[ color=blue] (6,5) -- (6,4);
      \draw[-to, color=blue] (6,4) -- (5,4);
      \draw[ color=blue] (5,4) -- (4,4);
      \draw[-to, color=blue] (1,6) -- (3,6);
      \draw[ color=blue] (3,6) -- (4,6);
      \draw[-to, color=blue] (2,4) -- (3.75,4);
      \draw[ color=blue] (3.75,4) -- (4,4);
      \draw [dash pattern=on 1.5pt off 1.5pt, gray] (0,0)-- (8,8);
    \end{tikzpicture}
   \end{subfigure}
  \caption{Graphics of the virtually unimodal combinatorics 
  $\ovra{\rho}_F(2)$ (left figure) and $\ovra{\rho}_F(3)$(right figure).}
  \label{figu:VU_comb_examples_for_l3_and_l4}
\end{figure}
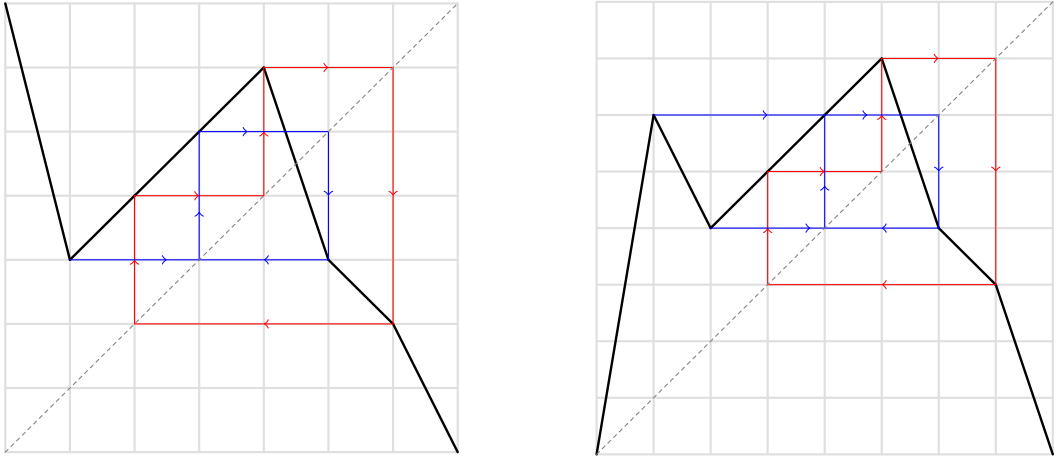

\section{Proof of the main theorem} \label{section:Proof_of_the_main_theorem}

In this section we will give the proof of the Main Theorem and Theorem \ref{theo:combinatorial_type_theorem}. 

Recall the following notation. For fixed $d \geq 3$, let 
$1 \leq \nu \leq d-1$, and 
$\boldsymbol{\mu} = (\mu_1,\mu_2, \ldots, \mu_{\nu}) \in \Z_+^{\nu}$, with 
$\sum_{i = 1}^{\nu}\mu_i = d-1$. We denote by $\cP_d^{\boldsymbol{\mu}}(\R)$
the set of real polynomials of degree $d$ with $\nu$ distinct real critical
points $c_1< c_2< \ldots< c_{\nu}$ and corresponding order 
$\mu_1, \mu_2, \ldots, \mu_{\nu}$.

Let $f \colon \R \longrightarrow \R$ be a real polynomial of degree 
$d$. We denote by $ \Crit (f) $ the set of critical points of $f$,
thus \[ \Crit (f) \= \{ x \in \R \colon f'(x) = 0 \}. \]

The \emph{real filled Julia set} of $f$, denoted by $K_{\R}(f)$, is the set
of all real numbers $x$ for which $\cO_f(x)$ is bounded. This set coincides 
with the real trace of the filled Julia set of $f$ as a polynomial
acting on $\C$. Note that $K_{\R}(f)$ is compact, 

\begin{equation}
    \label{eq:invariance_of_filled_Julia}
   f(K_{\R}(f)) \subseteq K_{R} \hspace{1cm} \text{and} \hspace{1cm}
f^{-1}(K_{\R}(f)) \subseteq K_{\R}(f). 
\end{equation}

By definition, the set of periodic points of $f$ is contained in 
$K_{\R}(f)$, then \[ \zeta_{f}(t) = \zeta_{f|_{K_{\R}(f)}}(t). \]

Given a periodic point $p$ of $f$ (or the cycle $\cO_{f} (p)$)  
its \emph{basin} is defined as the set 
\[ \cB_f(p) \= \{x \in \R \colon p \in \omega_{f}(x) \}. \] 
Observe that if $p$ is attracting, then its basin contains an open set. 
The \emph{immediate attracting basin} of $p$, denoted by
$\widehat{\cB}_f(p)$, 
is the union of the connected components of its basin intersecting 
$\cO_{f}(p)$. If $\widehat{\cB}_f(p)$ is a neighbourhood of $\cO_f(p)$ then
this orbit is called a \emph{two-sided attractor}, otherwise we will say
that is a \emph{one-sided attractor}.

The \emph{basin of $f$}
is the union of the basins of its attracting periodic points, and we denote 
it by $\cB(f)$. The \emph{repeller} of $f$ is the set
\[ \Lambda (f) \= K_{\R}(f) \setminus \cB (f). \] 

\subsection{Polynomial realization of virtually unimodal combinatorics} 
\label{subsect:polynomial_realization}
We start by proving that the virtually unimodal combinatorics built in the
last section can be realize by polynomial. 

Given a piecewise monotone combinatorics 
$\ovra{\rho} \in \{0,1, \ldots, n\}^{n+1}$, we can assign to each of its 
turning points $\tau_1 < \tau_2 < \ldots < \tau_{\nu}$, a local
degree $d_1, d_2, \ldots, d_{\nu}$ that is an even number. We call the 
pair $(\ovra{\rho},(d_1, \ldots, d_{\nu}))$ a \emph{graded combinatorics}.
We will say that a graded combinatorics 
$(\ovra{\rho},(d_1,\ldots,d_{\nu}))$
can be \emph{realized as a polynomial combinatorics} if there exists a real
polynomial map $P$ of degree $d= 1 + \sum_{i=1}^{\nu} (d_{i}-1)$ 
with $\nu$ real critical points $c_1, c_2, \ldots, c_{\nu}$ each of them 
with local degree $d_j$, so that $\ovra{\varrho}(P) = \ovra{\rho}$. Observe
that in this case, the multiplicity of each critical point $c_j$ is equal
to $\mu_j = d_j -1$.
\begin{prop}
    \label{prop:VU_PCF_polynomial_realization}
    For every $d \geq 3$, every $1 \leq \nu \leq d-1$, and every 
    $\boldsymbol{\mu} = (\mu_1,\mu_2, \ldots, \mu_{\nu}) \in \Z_+^{\nu}$
    with $\sum_{i = 1}^{\nu}\mu_i = d-1$, the set of post-critically finite
    virtually unimodal real polynomial map is non-empty in 
    $\cP_d^{\boldsymbol{\mu}}(\R)$.
\end{prop}

Observe that, in Section \ref{subsection_construction_VU_combinatorics}, 
we constructed a 
virtually unimodal combinatorics with $\nu$ distinct turning points for any
$\nu \geq  2$. So, to prove Proposition 
\ref{prop:VU_PCF_polynomial_realization}, is enough to prove that 
these virtually unimodal combinatorics are realizable by a post-critically finite 
polynomial of degree $d$.

For that purpose we need to start with some definitions. 
Let $\ovra{\rho} \in \{0,1, \ldots, n\}^{n+1}$ be a piecewise monotone 
combinatorics. We will say that $x \in \{0,1, \ldots, n\}$ is a 
\emph{Fatou point} or it is of \emph{Fatou type} if 
$\cO_{F_{\ovra{\rho}}}(x)$ contains a turning point. Else, we say it is 
a \emph{Julia point} or of \emph{Julia type.}
We will say that $\ovra{\rho}$ is an \emph{expanding combinatorics} 
(or just expanding) if for any $0 \leq j < n$ so that $j$ and $j+1$ are of
Julia type there exists $m(j) \geq 1$ for which
$|F_{\ovra{\rho}}^{m(j)}(j) - F_{\ovra{\rho}}^{m(j)}(j+1)| > 1$.

We make use of the following theorem that is a direct consequence of 
\cite[Theorem 1.1]{Porier2010_Hubbard_trees} (Compare with 
\cite{Douady-Hubbard1993_Top_charac_of_rationa_maps},
\cite{Bruin-Schleicher2008_Hubbard_trees_quadratic_polynomials},
\cite{Bonifan-Milnor_Surtherland2021_Thuston_alg_real_poly}). )
\begin{theo}
    \label{theo:polynomial_realization_Porier}
    A graded combinatorics $(\ovra{\rho},(d_1,\ldots,d_{\nu}))$
    can be realized as a polynomial combinatorics if and only if 
    $\ovra{\rho}$ is expanding.
\end{theo}

Lets consider the virtually unimodal combinatorics with $\nu$ 
turning points constructed in the previous section. 
We choose the periodic points
$k_{1,0}, k_{2,0}, \ldots, k_{\nu-1,0}$ in such a way that for every 
$i = 1, \ldots, \nu$ the points $k_{i,0}$ and $k_{i+1,0}$ have 
different itineraries (we can choose them as in 
Example \ref{example_VU_combinatorics}, then every turning point land at the
same cycle of period two, but two consecutive turning points have images at 
opposite sides of the dominant turning point). Lets call this combinatorics
$\ovra{\rho}_{\nu} \in \{ 0,1, \ldots, n \}^{n+1}$.

Since every real post-critically finite polynomial restricted to its filled 
Julia set is boundary anchored, we can consider framed combinatorics. 

\begin{lemm}
    \label{lem:expanding_VUC}
    The piecewise monotone combinatorics $\ovra{\rho}_{\nu}$ is expanding.
\end{lemm}

\begin{proof}
    First, we observe that the set of Fatou points of 
    $\ovra{\rho}_{\nu}$ is given by 
    \[ \cO_{F_{\ovra{\rho}_{\nu}}}(\rho_{\nu+1}) = \{ \rho_{\nu+1}, F_{\ovra{\rho}_{\nu}}(\rho_{\nu +1}), F^2_{\ovra{\rho}_{\nu}}(\rho_{\nu +1}) \}.\] 
    Then, the set of points of Julia type is given by
    $\{ 0,1, \ldots, n \} \setminus \cO_{F_{\ovra{\rho}_{\nu}}}(\rho_{\nu+1}).$ Since
    for every $1 \leq j \leq \nu - 1$ we have that
    $F_{\ovra{\rho}_{\nu}}(j) = k_{j,0}$, and the points 
    $k_{1,0}, k_{2,0}, \ldots, k_{\nu-1,0}$ have pairwise
    distinct itineraries, there exists $m(j) \geq 1$ such that
    \[F_{\ovra{\rho}_{\nu}}^{m(j)}(j) < F_{\ovra{\rho}_{\nu}}(\rho_{\nu +1})
    < F_{\ovra{\rho}_{\ell}}^{m(j)}(j+1),\] for every $1 \leq j \leq \nu-1.$
    Then, 
    \[ |F_{\ovra{\rho}_{\nu}}^{m(j)}(j)- 
    F_{\ovra{\rho}_{\nu}}^{m(j)}(j+1) \geq 2. \] Also, since 
    $\ovra{\rho}_{\nu}$ is framed, we have that 
    $F_{\ovra{\rho}_{\nu}}(0) \in \{0, n+1 \}$. Since, 
    $F_{\ovra{\rho}_{\nu}}(1) \in \langle \rho_{\nu}, \rho_{n-1} \rangle$,
    we have that 
    \[ F_{\ovra{\rho}_{\nu}}(0)- F_{\ovra{\rho}_{\nu}}(1) \geq 2. \]
    This proves that $\ovra{\rho}_{\nu}$ is expanding.
    
\end{proof}


\begin{proof}[ Proof of Proposition \ref{prop:VU_PCF_polynomial_realization}]
    Let $d \geq 3$. Consider the virtually unimodal combinatorics with $\nu$ turning points constructed before $\ovra{\rho}_{\nu}$.

    By Lemma \ref{lem:expanding_VUC} and Theorem 
    \ref{theo:polynomial_realization_Porier}, the graded combinatorics
    $(\ovra{\rho}_{\nu}, (\mu_1+1, \mu_2+1, \ldots, \mu_{\nu}+1))$
    can be realized by a post-critically finite real polynomial map. So
    there exists a post-critically
    finite real polynomial map $P_{\ovra{\rho}_{\nu}}$ with
    \[ \ovra{\varrho} (P_{\ovra{\rho}_{\nu}}) = \ovra{\rho}_{\nu}. \] 
    In particular, $P_{\ovra{\rho}_{\nu}}$ is virtually unimodal and it
    has $\nu$ distinct critical point (marked as the turning points of 
    $\ovra{\rho}_{\nu}$).

    Finally, by construction, the critical points 
    $c_1, c_2, \ldots, c_{\nu}$ have degree
    $\mu_1 + 1, \mu_2+1, \ldots, \mu_{\nu-1}+1$ respectively. So, each
    turning point of $P_{\ovra{\rho}_{\nu}}$, seeing as a critical point,
    have multiplicity $\mu_i$. In particular,
    the degree of $P_{\ovra{\rho}_{\nu}}$ is equal to 
    $ 1 +\sum_{i=1}^{\nu} \mu_i = d$. Thus, $P_{\ovra{\rho}_{\nu}}$ belongs
    to the set $\cP_d^{\boldsymbol{\mu}}(\R)$ .
\end{proof}


\subsection{Holomorphic motion of repellers} \label{subsect:Holomorphic_motion_of_repellers}

Let $P_0$ be a post-critically finite, virtually unimodal polynomial in 
$P_d^{\boldsymbol{\mu}}(\R)$ as constructed above. So, $P_0$ has $\nu$ 
critical points $c_{1,0}< c_{2,0}< \ldots < c_{\nu,0}$ of multiplicity
$\mu_1, \mu_2, \ldots, \mu_{\nu}$. Put 
\[ \vec{c}_0 = (c_{1,0}, c_{2,0}, \ldots, c_{\nu,0}), \hspace{0.3cm} 
\text{ and } \hspace{0.3cm}
\vec{v}_0 = (v_{1,0}, v_{2,0}, \ldots, v_{\nu,0}),\] where 
$v_{i,0} \= P_0(c_{i,0})$ for all $1 \leq i \leq \nu$. For 
$P \in P_d^{\boldsymbol{\mu}}(\R)$ put
\[ \vec{c}(P) = (c_{1}(P), c_{2}(P), \ldots, c_{\nu}(P)), \hspace{0.3cm} 
\text{ and } \hspace{0.3cm}
\vec{v}(P) = (v_{1}(P), v_{2}(P), \ldots, v_{\nu}(P)),\] as the vectors of
critical points and critical values of $P$ respectively, with
\[c_1(P) < c_2(P) < \ldots < c_{\nu}(P).\]

We start by defining a suitable set to study how the critical orbits of 
$P_0$ bifurcates. All the polynomials we will consider in this section 
will have critical points with bounded orbits. Thus, their filled Julia set 
is connected (observe that this is true for the virtually unimodal maps).
Moreover, conjugating with a real affine map, we may assume that the real
filled Julia set is equal to $[0,1]$. Let $\cP_0(\C)$ be the family of 
all complex polynomials $P$ of degree $d$ with $\nu$ distinct critical 
points $\{ c_1(P), \ldots, c_{\nu}(P)\}$ of multiplicity 
$\mu_1, \ldots, \mu_{\nu}$ respectively (when the critical points are all
real we will consider them in increasing order), with $P(0)= P_0(0)$ and
$P(1)=P_0(1)$ (if $\Crit (P) \subset \R$ this last condition implies that 
$P$ has the same shape as $P_0$). 
We denote by $\cP_0(\R)$ the 
subfamily of $\cP_0(\C)$ consisting of polynomials with real coefficients
and $\cU$ the subfamily of $\cP_0(\R)$ of those maps with only real 
critical points.

It is well known that $\cP_0(\C)$ is a complex manifold
of dimension $\nu$, and the critical values are holomorphic coordinated
(See \cite{Levin_Multipliers_of_periodic_orbits}). We will use the following
proposition.

\begin{prop}[\cite{Levin_Shen_van_Strien_Transversality_elementary_proof}]
    There exists a neighborhood $W$ of $\vec{v}_0$ in $\C^{\nu}$ such that 
    the map $F \colon W \longrightarrow \cP_0(\C)$ 
    assigning to each $w \in W$ a polynomial 
    $P_w \in \cP_0(\C)$ with $\vec{v}(P_w) = w$, is a 
    biholomorhpism onto a neighborhood of $P_0$. In particular, the critical
    points of $P_w$ depend holomorphically on $w$.
\end{prop}

In the previous proposition, we can consider $W$ to be a polydisk centered
at $\vec{v}_0 \in \R^{\nu}$ and radius some $\e > 0$ small enough, whose
image is an open neighborhood of $P_0$. In this case, the cube 
\[ C(\vec{v}_0, \e) \= \{ x\in \R^{\nu} \colon |x_i - v_{i,0}| < \e  
\text{ for all } i = 1,2, \ldots, \nu \}, \] is contained in $W$, and 
$F(W) \cap \cP_0(\R) \neq \emptyset$. Now, let 
$\cU \subset \C(\vec{v}_0,\e)$ be 
so that $P_w \in \cP_0(\R)$ for every $w \in \cU$. 
Now, for each $w \in \cU$ every $0 \leq i < \nu$, put 
$I_{i,w} \= [c_i(P_w),c_{i+1}(P_w)]$ where 
$c_0(P_w) =0$ and $c_{\nu+1}(P_w) =1$. Then, since for $P_0$ all the 
critical values are in $(P_0^2(c_{\nu,0}, P_0(c_{\nu,0}))$, we can consider
$W$ small enough so that, for every $w \in \cU$ the map $P_w$ is virtually 
unimodal with dominant turning point $c_{\nu}(P_w)$, and it has an
attracting periodic cycle of period three
in $I_{\nu-1,w} \cup I_{\nu,w}$, for which the interior of the convex hull
of its orbit contains the single critical points $c_{\nu}(P_w)$.

Lets denote by $I_w$ the convex hull of this attracting periodic orbit of 
period three. So we have $I_w \subset I_{\nu-1,w} \cup I_{\nu,w}$.


With this considerations we will study how the repeller 
$\Lambda (P_0) \cap (I_{\nu-1} \cup I_{\nu})$ deforms in the neighborhood
$F(\cU)$ of $P_0$ in $\cP_0(\R)$.
Let $E$ be a subset of $\CC$. A \emph{holomorphic motion} of $E$ parametrized
by a pointed complex manifold $(W, \boldsymbol{v}_0)$ is a function
\[ h \colon W \times E \longrightarrow \CC \] such that 
\begin{enumerate}
    \item $h(\boldsymbol{v}_0,z) = z$ for all $z \in E$,
    \item for every $w \in W$, the map $z \mapsto h(w,z)$ is
    injective, and 
    \item for every $z \in E$ the map $w \mapsto h(w,z)$ is
    holomorphic for $w \in W$.
\end{enumerate}


\begin{prop}
    \label{prop:conjugacy_with_Fib_SFT}
    For every $w \in \cU$, the set $\Lambda(P_w) \cap I_w$ is a 
    Cantor set and the action of $P_w$ restricted to it is topologically 
    conjugated to the action of the left shift map $\sigma$ on $\SFib$. 
    Moreover, there is an holomorphic motion $h$ of 
    $\Lambda(P_0) \cap (P_0^2(c_{\nu,0}),P_0^2(c_{\nu,0}))$ defined on 
    $(W,\vec{v}_0)$ such that for every $w \in \cU$ we have that
    \begin{equation}
        \label{eq:real_holomorphic_motion}
        h(w,\Lambda(P_0) \cap (P_0^2(c_{\nu,0}),P_0^2(c_{\nu,0}))
    = \Lambda(P_w) \cap I_w .
    \end{equation}
\end{prop}

\begin{proof}
    Recall that the
    the shape of $P_0$ restricted to 
    $I_{\nu -1,\vec{v}_0} \cup I_{\nu,\vec{v}_0}$ is $(+,-)$.
    In this case, we have $ P_0^2(c_{\nu,0}) < c_{\nu,0} < P_0(c_{\nu,0})$,
    and $I_{\vec{v}_0}= [P_0^2(c_{\nu,0}), P_0(c_{\nu,0})]$. 
    Put $I_1 \= [P_0(c_{\nu,0}), c_{\nu,0}]$, and
    $I_2 \= [c_{\nu,0}, P_0(c_{\nu,0})]$, so $P_0(I_1) = I_2$, and 
    $P_0(I_2) = I$. We will say that a point
    $x \in I_i \setminus \cO_{P_0}(c_{\nu,0})$ has a 
    \emph{symmetric with respect to $c_{\nu,0}$}, or just that it has a symmetric,
    if there is a $\hat{x} \in I \setminus I_i$ such that 
    $P_0(\hat{x}) = P_0(x)$. Observe that every 
    $x \in \interior (I_1)$ 
    has a symmetric, but if $\tilde{c}_{\nu,0} \in I_2$ is so that 
    $P_0(\tilde{c}_{\nu,0}) = c_{\nu,0}$, then no point in 
    $[\tilde{c}_{\nu,0} , P_0(c_{\nu,0}))$ has a symmetric.

    The immediate basin of attraction of 
    $\cO_{P_0}(c_{\nu,0})$ restricted to $I_{\vec{v}_0}$ is the union of the
    three intervals $[P_0^2(c_{\nu,0}), p_2)$, $(\hat{p}_0, p_0)$, and 
    $(p_1, P_0(c_{\nu,0})]$, 
    where $p_0$ is a periodic point of period three with 
    $P_0(p_0) = p_1$, and $P_0^2(p_0) = p_2$, and 
    $\hat{p}_0$ is the symmetric of $p_0$. 

    First, we will construct the inverse branches of $P_0$ and $P_0^2$
    on a suitable subinterval of $I_{\vec{v}_0}$. This construction will
    extend to $P_w$ with $w$ in $\cU$.

    Since $c_{\nu,0}$ is periodic of period three, 
    $\Crit (P_0) \cap \interior(I_{\vec{v}_0}) = \{c_{\nu,0}\}$, and 
    $\partial I_{\vec{v}_0} \subset \cO_{P_0}(c_{\nu,0})$, we have that, 
    $P_0^3$ is a bijection from the interval $[P_0^2(c_{\nu,0}), p_2]$ onto
    itself, and $P_0^2(c_{\nu,0})$ is an attracting fixed point of $P_0^3$.
    Since $[P_0^2(c_{\nu,0}), p_2)$ is contained in 
    $\widehat{\cB}_{P_0}(c_{\nu,0})$,
    there are no fixed points of $P^3_0$ in its interior. So
    for all $x \in (P^2_0(c_{\nu,0}), p_2) $, we have that $P_0^3(x) < x$.

    Fix $\ell_0 \in (P_0^2(c_{\nu,0}),p_2)$. We have 
    \[ P_0(\ell_0) \in (c_{\nu,0}, p_0), \hspace{0.6cm} 
    P_0^2(\ell_0) \in (p_1,P_0(c_{\nu,0})), \hspace{0.3cm}
    \text{ and } \hspace{0.3cm} 
    P_0^3(\ell_0) \in (P_0^2(c_{\nu,0}), \ell_0). \]
    Let $\ell_1 \in (P_0^3(\ell_0), \ell_0)$. Then, 
    \[P_0(\ell_1) \in (c_{\nu,0}, P_0(\ell_0)), \hspace{0.5cm}  \text{ and }
    \hspace{0.5cm}
    P_0^2(\ell_1) \in (P_0^2(\ell_0), P_0(c_{\nu,0})).\]

    Let $\widehat{P_0(\ell_0)}$ be the symmetric of $P_0(\ell_0)$, and put 
    \[ J_{\vec{v}_0} \= (P_0^3(\ell_0), P_0^2(\ell_1)),  \hspace{0.5cm}
    J_{1,\vec{v}_0} \= (\ell_1, \widehat{P_0(\ell_0)}), \hspace{0.5cm} 
    \text{ and } 
    \hspace{0.5cm} J_{2,\vec{v}_0} \= (P_0(\ell_1), P_0^2(\ell_0)).\] Then, 
    $\overline{J_{i,\vec{v}_0}} \subset J_{\vec{v}_0}$, 
    $c_{\nu,0} \notin \overline{J_{i,\vec{v}_0}}$ for 
    $i = 1,2$, 
    $\overline{J_{1,\vec{v}_0}} \cap \overline{J_{2,\vec{v}_0}} = \emptyset$, 
    $P_0(J_{1,\vec{v}}) = J_{2,\vec{v}_0}$, and 
    $P_0(J_{2,\vec{v}_0}) = J_{\vec{v}_0}$, see Figure \ref{fig:intervals_Js}.

    Since \[\overline{J_{i,\vec{v}_0}} \cap (\Crit (P_0) \cup 
    P_0(\Crit (P_0)) = \emptyset,\] 
    the restrictions 
    \[ P_0|_{J_{2,\vec{v}_0}} \colon J_{2,\vec{v}_0} \longrightarrow J_{\vec{v}_0} \hspace{0.5cm} 
    \text{ and } \hspace{0.5cm} P_0^2|_{J_{1,\vec{v}_0}} \colon J_{1,\vec{v}_0} \longrightarrow
    J_{\vec{v}_0},\]
    are diffeomorphisms. So, the inverse branches of $P_0$ and $P_0^2$
    restricted to $J_{\vec{v}_0}$, are well defined. Put 
    \[ \phi_{0,2} \colon J_{\vec{v}_0} \longrightarrow J_{2,\vec{v}_0} \hspace{0.5cm} 
    \text{ and } \hspace{0.5cm} \phi_{0,1} \colon J_{\vec{v}_0} \longrightarrow
    J_{1,\vec{v}_0},\] with $\phi_{0,1} \= (P_0^2|_{J_{1,\vec{v}_0}})^{-1}$,
    and $\phi_{0,2} \= (P_0|_{J_{2,\vec{v}_0}})^{-1}$. Also, since the critical 
    points of $P_0$ are all real, the inverse branches, $\phi_{0,2}$ and $\phi_{0,1}$ extend uniquely to $\C_{J_{\vec{v}_0}}$. So 
    \[ \phi_{0,1} \colon \C_{J_{\vec{v}_0}} \longrightarrow \C_{J_{1,\vec{v}_0}} \hspace{0.5cm} 
    \text{ and } \hspace{0.5cm} \phi_{0,2} \colon \C_{J_{\vec{v}_0}} 
    \longrightarrow \C_{J_{2,\vec{v}_0}},\] are holomorphic. 

    Fix $\theta \in (0, \pi)$. 
    By the definition of $J_{1,\vec{v}_0}$ and $J_{2,\vec{v}_0}$, we have
    that 
    \[ \phi_{0,1}(J) = J_{1,\vec{v}_0} \subset J_{\vec{v}_0} \hspace{0.5cm}
    \text{ and } \hspace{0.5cm} 
    \phi_{0,2}(J_{\vec{v}_0}) = J_{2,\vec{v}_0} \subset J_{\vec{v}_0}.\] Then, by
    Lemma \ref{lemm:Schwarz_Lemma}, we have that 
    \begin{equation}
        \label{eq:Schwarz_Lemma_hypothesis_for_phi_{s_0,i}}
        \phi_{0,i}(D_{\theta}(J_{\vec{v}_0})) \subset D_{\theta}(J_{i,\vec{v}_0}) \subset 
        D_{\theta}(J_{\vec{v}_0}). 
    \end{equation}
    Moreover, since $\overline{J_{i,\vec{v}_0}} \subset J_{\vec{v}_0}$, 
    we have
    \[ \overline{D_{\theta}(J_{i,\vec{v}_0})} \subset D_{\theta}(J_{\vec{v}_0}). \] So   
    \begin{equation}
    \label{eq:Schwarz-Pick_hypohtesis_for_phi_{s_0,i}}
    \overline{\phi_{0,i}(D_{\theta}(J_{\vec{v}_0}))} \subset \overline{D_{\theta}(J_{i,\vec{v}_0})}
    \subset D_{\theta}(J_{\vec{v}_0}).
    \end{equation}
    Consider open intervals $K_{1,\vec{v}_0}$ and $K_{2,\vec{v}_0}$ satisfying the following:
    \begin{enumerate}
        \item $\overline{K_{i,\vec{v}_0}} \subset J_{\vec{v}_0}$, for 
        $i = 1,2$;
        \item $\overline{J_{i,\vec{v}_0}} \subset K_{i,\vec{v}_0}$, for 
        $i =1,2$;
        \item $\overline{K_{1,\vec{v}_0}} \cap \overline{K_{2,\vec{v}_0}} = \emptyset$;
        \item $ c_{\nu,0} \notin \overline{K_{1,\vec{v}_0}} \cup \overline{K_{2,\vec{v}_0}}$.
    \end{enumerate}
    In particular, for $i = 1,2$
    \[\overline{D_{\theta}(J_{i,\bf{v}_0})} \subset D_{\theta}(K_{i,\vec{v}_0}) \hspace{0.5cm}
    \text{ and } 
    \hspace{0.5cm} \overline{D_{\theta}(K_{i,\vec{v}_0})} \subset D_{\theta}(J_{\vec{v}_0}).\]
    So, by \eqref{eq:Schwarz-Pick_hypohtesis_for_phi_{s_0,i}}
    \begin{equation}
        \label{eq:Image_of_D_theta_K_i}
        \overline{D_{\theta}(J_{\vec{v}_0})} \subset P_0^2(D_{\theta}(K_{1,\vec{v}_0})) 
        \hspace{0.5cm} \text{ and } \hspace{0.5cm}
        \overline{D_{\theta}(J_{\vec{v}_0})} \subset P_0(D_{\theta}(K_{2,\vec{v}_0})).
    \end{equation}    

    Using that  
    \[ (\Crit (P_0) \cup P_0(\Crit (P_0)) \cap 
    \left(\overline{D_{\theta}(K_{1,\vec{v}_0})} \cup \overline{D_{\theta}(K_{2,\vec{v}_0})} \right) 
    = \emptyset,  \] \eqref{eq:Image_of_D_theta_K_i}, and the Open Mapping 
    Theorem, we can choose $W$ to be small enough so that 
    for every $w \in W$ the previous construction carry over. Denoting
    the corresponding sets $I_w$, $J_w$, $K_{1,w}$, and $K_{2,w}$ for every 
    $w \in W$, we have
    \begin{equation}
        \label{eq:f_s_containsD(J)}
      \overline{D_{\theta}(J_w)} \subset P_w^2(D_{\theta}(K_{1,w})) 
    \hspace{0.5cm} \text{ and } \hspace{0.5cm}
    \overline{D_{\theta}(J_w)} \subset P_w(D_{\theta}(K_{2,w})),  
    \end{equation}
    and
    \begin{equation}
        \label{eq:no_crit_in_D(K_i)}
        (\Crit (P_w) \cup P_w(\Crit (P_w)) \cap 
    \left(\overline{D_{\theta}(K_{1,w)}} \cup \overline{D_{\theta}(K_{2,w})} \right) 
    = \emptyset.
    \end{equation}
    Additionally, if $w \in \cU$, then 
    $I_w \setminus \ov{K_{1,w}} \cup 
    \ov{K_{2,w}}$ intersects a unique attracting cycle of period three and 
    $\ov{K_{1},w}$ and $\ov{K_{2,w}}$ intersect each component of the 
    immediate basin
    of this attracting cycle. These two conditions are to ensure that, 
    for all $w \in \cU$, we have 
    $\Lambda (P_w) \cap I_w \subset K_{1,w} \cup K_{2,w}$, thus 
    \[ \Lambda (P_w) \cap I_w = \{ x \in I_w \colon P_w^n(x) \in K_{1,w}
    \cup K_{2,w} \text{ for all } n \geq 0 \}.\]

    Conditions \eqref{eq:f_s_containsD(J)} and \eqref{eq:no_crit_in_D(K_i)}
    allow us to define the inverse branches of $P_w$ and 
    $P_w^2$ restricted to $D_{\theta}(K_{1,w})$ and $D_{\theta}(K_{2,w})$ 
    respectively, for every $w \in W$.

    Put 
    \[ \phi_{w,1} \= (P_w^2|_{D_{\theta}(K_{1,w})})^{-1} \colon D_{\theta}
    (J_w) \longrightarrow D_{\theta}(K_{1,w}) \]  and 
    \[ \phi_{w,2} \= (P_w|_{D_{\theta}(K_{2,w})})^{-1} \colon D_{\theta}(J_w)
    \longrightarrow D_{\theta}(K_{2,w}).\]
    Now, for $i = 1,2$, and every $w \in W$, we have 
    \begin{equation}
    \label{eq:Unif_Cauchy_hyp}
    \overline{\phi_{w,i}(D_{\theta}(J_w))} \subset \overline{D_{\theta}
    (K_{i,w})} \subset D_{\theta}(J_w).
    \end{equation}
    Thus, the image of $D_{\theta}(J_w)$
    under $\phi_{w,i}$ is compactly contained in $D_{\theta}(J_w)$. Then, 
    for every $w \in W$, the maps $\phi_{w,1}$ and $\phi_{w,2}$ 
    satisfies the hypothesis of Lemma \ref{lemm:uniform_contraction}. So, 
    for every $\bf{u} = (u_i)_{i \geq 0} \in \{ 1,2 \}^{\N_0}$ and 
    $k \geq 0$, if we put 
    \[\phi_{w,\bf{u}|_{[0,k]}} \= \phi_{w,u_0} \circ \phi_{w,u_1} \circ 
    \ldots \circ \phi_{w,u_k},\] then 
    $\{ \phi_{w,\bf{u}|_{[0,k]}} \}_{k \geq 0}$ converges
    uniformly on $D_{\theta}(J_w)$ to a constant function. We write 
    $\phi_{w,\bf{u}}$ for its limit, and $x_w(\bf{u})$ for the value of 
    $\phi_{w,\bf{u}}$ in $D_{\theta}(J_w)$.

    Now we will study how these inverse branches behave under composition.
    First, we define the map
    $\vee \colon \cL_{\text{Fib}} \longrightarrow \{ 1,2\}^*$ in the 
    following way: If the length of $u \in \cL_{\text{Fib}}$ is one (thus 
    $u \in \{1,2\}$), $\vee(u) = u$.
    For $u = u_0u_1 \ldots u_n \in \cL_{\text{Fib}}$ with $n > 1$, we
    read it from left to right, if $i < n-1$ and $u_i =1$, then $u_{i+1} =2$
    and $\vee$ replaces $u_iu_{i+1}$ by $1$, and we continue reading 
    $u_{i+2}$. If $u_i = 2$ then $\vee$ keeps the symbol and continue
    reading $u_{i+1}$. The map $\vee$ stops once it moves beyond the last
    symbol. If $\vee$ reads $u_n$, then it keeps it.

    Fix $w \in W$ and $\bf{u} = (u_i)_{i \geq 0} \in \SFib$. 
    Since, for $i \in \{ 1,2\}$, and every $x \in D_{\theta}(J_w)$ we have 
    that $\phi_{w,i}(x) \in D_{\theta}(K_{i,w})$, and 
    $P_w(\phi_{w,1}(x)) \in D_{\theta}(K_{2,w})$, it follows that, for every 
    $k \geq 0$
    \begin{equation}
    \label{eq:finite_itineraries_using_inverse_branches}
    \phi_{w,\vee(\bf{u}|_{[0,k]})}(D_{\theta}(J_w))=\{ x \in D_{\theta}(J_w)
    \colon 
    P_w^i(x) \in D_{\theta}(J_{u_i,w}) \text{ for } i = 0,1, \ldots, k \}.
    \end{equation}
    Moreover, since for every $k,n \in \N_0$ with $n \leq k$ we have 
    \begin{equation}
        \vee(\bf{u}|_{[0,n]})|_{\{i\}} = \vee(\bf{u}|_{[0,k]})|_{\{i\}},
    \end{equation}
    for $i = 0,1, \ldots, |\vee(\bf{u}|_{[0,n]})|-1$. Then
    \begin{align*}
        \phi_{w,\vee(\bf{u}|_{[0,k]})}(D_{\theta}(J_w)) & = 
        \{ x \in D_{\theta}(J_w) \colon P_0^i(x) \in D_{\theta}(K_{u_i,w}) 
        \text{ for } i = 1,2, \ldots, k-1 \} \\
        &\subseteq \{ x \in D_{\theta}(J_w) \colon P_0^i(x) \in 
        D_{\theta}(K_{u_i,w}) \text{ for } i = 1,2, \ldots, n-1 \} \\
        &= \phi_{s, \vee(\bf{u}|_{[0,n]})}(D_{\theta}(J_w)).
    \end{align*}
    In particular, for all $n \in \N_0$, we have
    \begin{equation}
        \label{eq:nested_sequence}
        \phi_{w,\vee(\bf{u}|_{[0,n+1]})}(D_{\theta}(J_w)) \subseteq 
        \phi_{w,\vee(\bf{u}|_{[0,n]})}(D_{\theta}(J_w)).
    \end{equation}

    Now we prove that the action of $P_0$ restricted to 
    $\Lambda (P_w) \cap I_0$ is conjugated to the action of $\sigma$ on 
    $\SFib$. First, observe that 
    $I_0 \setminus (K_1 \cup K_2)$ contains $\cO_{P_0}(c)$, it has nonempty
    interior and is contained in the immediate basin of 
    $\cO_{P_0}(c_{\nu,0})$.
    It follows that  \[ (\Lambda(P_0) \cap I) \subset K_1 \cup K_2. \]
    Thus, $x \in \Lambda (P_0) \cap I_0$ if and only if 
    $\cO_{P_0}(x) \subset K_1 \cup K_2.$ Then, for every 
    $n \geq 1$
    \begin{equation}
    \label{eq:contention_of_the_repppeller_under_inverse_branches}
        (\Lambda(P_0) \cap I_0) \subset \bigcup_{u \in \cL_{\text{Fib}(n)}} 
        \phi_{0,\vee(u)} (J_0),
    \end{equation}
    Thus $x_{0}(\bf{u}) \in K_1 \cup K_2$. By 
    \eqref{eq:contention_of_the_repppeller_under_inverse_branches}, 
    $x_{0}(\bf{u}) \in \Lambda(P_0) \cap I_0$, and by 
    \eqref{eq:finite_itineraries_using_inverse_branches} we have that the 
    itinerary of $x_{0}(\bf{u})$ under $P_0$ is equal to 
    $(I_{u_i})_{i \geq 0}$. So we can define the map
    \begin{align*}
        \psi \colon \SFib &\longrightarrow \Lambda(P_0) \cap I \\
        \bf{u} &\mapsto x_{0}(\bf{u}).
    \end{align*}
    This map is injective since if $\bf{u}, \bf{z} \in \SFib$ with 
    $\bf{u} \neq \bf{z}$,
    then $u_i \neq z_i$ for some $i \geq 0$. So, 
    $J_{u_i} \cap J_{z_i} = \emptyset$, and  
    $P_0^i(x_{0}(\bf{u})) \in J_{u_i}$ and 
    $P_0^i(x_0(\bf{z})) \in J_{z_i}$. This implies that 
    $P_0^i(x_{0}(\bf{u})) \neq P_0^i(x_{0}(\bf{z}))$. It is 
    onto since for any $x \in \Lambda (P_0) \cap I_0$ with itinerary 
    $(I_{u_i})_{i \geq 0}$, then $(u_i)_{i \geq 0} \in \SFib$, and 
    \begin{equation}
        \label{eq:identyti_of_the_itinerary}
        \cI(x_{0}((u_i)_{i \geq 1})) = (I_{u_i})_{i \geq 0} = \cI(x).
    \end{equation}
    This implies that $x_{0}((u_i)_{i \geq 0}) = x$. Thus, $\psi$ is a bijection.
    
    To prove the continuity of $\psi$, first observe that 
    \eqref{eq:contention_of_the_repppeller_under_inverse_branches} implies 
    that for $u = u_0 \ldots u_{n-1} \in \cL_{\text{Fib}}(n)$ we have 
    \[P_0^i(\phi_{0,\vee(u)}(J)) \subset J_{u_i},\] for 
    $i = 0, \ldots, n-1.$
    If $u = u_0 \ldots u_{n-1}$ and $z = z_0, \ldots, z_{n-1}$ are in 
    $\cL_{\text{Fib}}(n)$ with $u \neq z$, thus $u_i \neq z_i$ for some 
    $i \in \{ 0,1, \ldots, n-1 \}$, then 
    \[ P_0^i(\phi_{0,\vee(u)}(J_0)) \cap P_0^i(\phi_{0,\vee(z)}(J_0))
    = \emptyset.\] This implies that 
    $\phi_{0,\vee(u)}(J_0) \cap \phi_{0,\vee(z)}(J_0) = \emptyset$. So the 
    number of connected components of 
    \[ \bigcup_{u \in \cL_{\text{Fib}}(n)} \phi_{0,\vee(u)}(J_0)  \] is equal
    to the nth Fibonacci number $\ell_n$. Put 
    \[ d_n \= \max_{u \in \cL_{\text{Fib}(n)}} \diam_{D_{\theta}(J_0)}
    (\phi_{0,\vee(u)}(J_0)).\] 
    Since for every $\bf{u} \in \SFib$ the sequence 
    $\{ \phi_{s,\bf{u}|_{[0,k]}} \}_{k \geq 0}$ converges uniformly on 
    $D_{\theta}(J_0)$ to a constant function, it follows that
    $d_n \rightarrow 0$ as 
    $n \rightarrow \infty$. Let $\e > 0$. Choose $n$ large enough
    so that $d_n < \e$. Taking $\delta = 2^{-n}$, we have that
    any two sequences $\bf{u}, \bf{z} \in \SFib$
    with $\dist(\bf{u},\bf{z}) < 2^{-n}$ belong to a cylinder
    $[a]$ with $a \in \cL_{\text{Fib}}(n)$. Then 
    \[ \dist_{D_{\theta}(J_0)}(\phi_{0,\bf{u}}, \phi_{s_0,\bf{z}}) < 
    d_n < \e. \]
    Thus, $\psi$ is continuous.

    Since $\SFib$ and $\Lambda(P_0) \cap I$ are compact spaces, and
    $\psi$ is a continuous bijection, we have that $\psi^{-1}$ is also 
    continuous.

    Finally, observe that by \eqref{eq:finite_itineraries_using_inverse_branches}, it follows that 
    for any $\bf{u} \in \SFib$ we have 
     $\psi(\sigma (\bf{u})) = P_0(\psi(\bf{u}))$. So $\psi$ is a
     topological conjugacy between the action of $P_0$ on 
     $\Lambda (P_0) \cap I_0$ and $\SFib$.

    The construction above carries without changes to $P_w$ for 
    $w \in \cU$. Since we chose $W$ small enough 
    so that for $w \in \cU$ the map $P_w$ has a periodic
    attracting
    cycle of minimal period three close enough to $\cO_{P_0}(c_{\nu}(P_w))$
    so that $\Lambda (P_w) \cap I_w$ is contained in $K_{1,w} \cup K_{2,w}$
    and $\Crit (P_w) \cap (K_{1,w} \cup K_{2,w}) = \emptyset$.

    Now we will construct the holomorphic motion of 
    $\Lambda (P_0) \cap I_{\nu-1}\cup I_{\nu}$. Consider the map
    \[ h \colon W \times (\Lambda (P_0) \cap I_{\nu-1}\cup I_{\nu}) \longrightarrow \CC, \] 
    defined 
    as $h(w,x) \= \phi_{w,\vee(\cI(x))}$ for every 
    $(w,x) \in W \times (\Lambda (P_0) \cap I_{\nu-1}\cup I_{\nu})$.
    We have to prove that
    \begin{enumerate}
        \item $h(\vec{v}_0,x) = x$ for every $x \in \Lambda(P_0) \cap I_{\nu-1}\cup I_{\nu}$,
        \item $x \mapsto h(w,x)$ is injective for every $w \in W$, and 
        \item $w \mapsto h(w,x)$ is holomorphic for any $x \in \Lambda (P_0)\cap I_{\nu-1}\cup I_{\nu}$.
    \end{enumerate}

    For (1), observe that by \eqref{eq:identyti_of_the_itinerary}, for every 
    $x \in \Lambda (P_0) \cap I_{\nu-1}\cup I_{\nu}$, we have 
    $h(\vec{v}_0,x) = \phi_{\vec{v}_0,\cI(x)} = x.$

    To prove (2), fix $w \in W$. By definition, for every 
    $x \in \Lambda (P_0) \cap I_{\nu-1}\cup I_{\nu}$ we have 
    $\cI(h(w,x)) = \cI(\phi_{w,\cI(x)}) = \cI(x)$. Since for 
    $x, y \in \Lambda (P_0) \cap I_{\nu-1}\cup I_{\nu}$, with $x \neq y$, we have 
    $\cI(x) \neq \cI(y)$, the above implies that $h(w,x) \neq h(w,y)$.

    For (3), fix $x \in \Lambda (P_0) \cap I_{\nu-1}\cup I_{\nu}$. By Montel's Theorem, the
    family \[ \cF \= \{ \phi_{w, \vee(u)} \colon u \in \cL_{\text{Fib}} \, 
    \text{ and } w \in W \} \]
    is normal, and each $\phi_{w, \vee(u)}$ is holomorphic in $w$, we have 
    that every limit point of $\cF$ depends holomorphically on $w$. In 
    particular, for fixed $x \in \Lambda (P_0) \cap I_{\nu-1}\cup I_{\nu}$, the map
    $\phi_{w, \cI(x)}$ is holomorphic in $w$.

    Thus, $h$ is a holomorphic motion of $\Lambda (P_0)$ parametrized by 
    $W$. 

    Finally, equation \ref{eq:real_holomorphic_motion} follows from the 
    definition of the holomorphic motion.
    This conclude the proof of the proposition.
    
\end{proof}


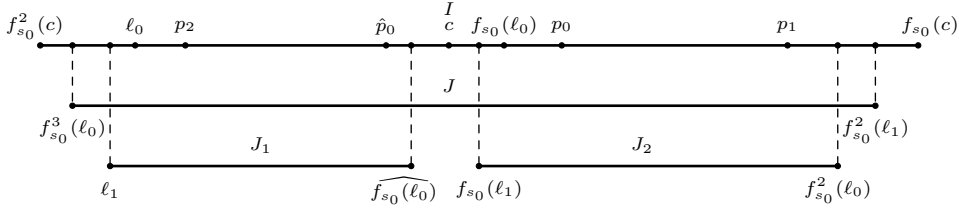
\begin{figure}[h!]
  \centering
    \begin{tikzpicture}[line cap=round,line join=round,>=triangle 45,x=3.3cm,y=4cm]
    \clip(-1.7606566481216794,-.5290389905687337) rectangle (2.06076906654733,0.3050905534000426);
\draw [line width=1pt] (-1.6283250703999999,0)-- (1.87052,0);
\draw [line width=1pt] (-1.5,-0.2)-- (1.7,-0.2);
\draw [line width=1pt] (-1.35,-0.4)-- (-0.15,-0.4);
\draw [line width=1pt] (0.12,-0.4)-- (1.55,-0.4);
\draw [line width=0.5pt, dashed] (-1.5,0)-- (-1.5,-0.2);
\draw [line width=0.5pt, dashed] (1.7,0)-- (1.7,-0.2);
\draw [line width=0.5pt, dashed] (-1.35,0)-- (-1.35,-0.4);
\draw [line width=0.5pt, dashed] (-0.15,0)-- (-0.15,-0.4);
\draw [line width=0.5pt, dashed] (0.12,0)-- (0.12,-0.4);
\draw [line width=0.5pt, dashed] (1.55,0)-- (1.55,-0.4);

\begin{scriptsize}
\draw [fill=black] (-1.6283250703999999,0) circle (1pt);
\draw[color=black] (-1.6505894678493178,0.06325904625025687) node {$f^2_{s_0}(c)$};

\draw [fill=black] (-0.25,0) circle (1pt);
\draw[color=black] (-0.25,0.06325904625025687) node {$\hat{p}_0$};

\draw [fill=black] (1.87052,0) circle (1pt);
\draw[color=black] (1.917382237811931,0.06325904625025687) node {$f_{s_0}(c)$};

\draw [fill=black] (-1.05,0) circle (1pt);
\draw[color=black] (-1.05,0.06325904625025687) node {$p_2$};

\draw [fill=black] (-1.25,0) circle (1pt);
\draw[color=black] (-1.25,0.06325904625025687) node {$\ell_0$};

\draw [fill=black] (-1.35,0) circle (1pt);
\draw [fill=black] (-1.35,-0.4) circle (1pt);
\draw[color=black] (-1.35,-0.48) node {$\ell_1$};
\draw[color=black] (-0.75,-0.33) node {$J_1$};

\draw [fill=black] (-1.5,0) circle (1pt);
\draw [fill=black] (-1.5,-0.2) circle (1pt);
\draw[color=black] (-1.5,-0.28) node {$f^3_{s_0}(\ell_0)$};

\draw [fill=black] (1.7,0) circle (1pt);
\draw [fill=black] (1.7,-0.2) circle (1pt);
\draw[color=black] (1.7,-0.28) node {$f^2_{s_0}(\ell_1)$};

\draw [fill=black] (-0.15,0) circle (1pt);
\draw [fill=black] (-0.15,-0.4) circle (1pt);
\draw[color=black] (-0.18,-0.48) node {$\widehat{f_{s_0}(\ell_0)}$};

\draw [fill=black] (0.45,0) circle (1pt);
\draw[color=black] (0.45,0.06325904625025687) node {$p_0$};

\draw [fill=black] (0.22,0) circle (1pt);
\draw[color=black] (0.22,0.06325904625025687) node {$f_{s_0}(\ell_0)$};

\draw [fill=black] (0.12,0) circle (1pt);
\draw [fill=black] (0.12,-0.4) circle (1pt);
\draw[color=black] (0.16,-0.48) node {$f_{s_0}(\ell_1)$};
\draw[color=black] (0.77,-0.33) node {$J_2$};

\draw [fill=black] (1.35,0) circle (1pt);
\draw[color=black] (1.35,0.06325904625025687) node {$p_1$};

\draw [fill=black] (1.55,0) circle (1pt);
\draw [fill=black] (1.55,-0.4) circle (1pt);
\draw[color=black] (1.55,-0.48) node {$f^2_{s_0}(\ell_0)$};

\draw[color=black] (0,0.06325904625025687) node {$c$};
\draw[color=black] (0,0.12325904625025687) node {$I$};
\draw[color=black] (0,-0.13325904625025687) node {$J$};
\draw [fill=black] (0,0) circle (1pt);
\end{scriptsize}
\end{tikzpicture}
   \caption{Diagram of the intervals 
   $I = [f_{s_0}^2(c), f_{s_0}(c)]$ on top, 
   $J = [f_{s_0}^3(\ell_0), f_{s_0}^2(\ell_1)]$ on the second level,
   $J_1 = [\ell_1, \widehat{f_{s_0}(\ell_o)}]$, and 
   $J_2 = [f_{s_0}(\ell_1), f_{s_0}^2(\ell_0)]$ on the third level.}
  \label{fig:intervals_Js}
\end{figure}

\begin{rema}
    Observe that in the proof of Proposition 
    \ref{prop:conjugacy_with_Fib_SFT} we did not use the fact that the 
    map $P_0$ is post-critically finite. Indeed, the result holds true for 
    any virtually unimodal polynomial whose dominant turning point is in 
    the basin of an attracting cycle.
\end{rema}

\subsection{Proof of the Main Theorem} \label{subsect:proof_of_the_main_theorem}
Let $P_0$ be as before. By construction, the critical values 
$v_{1,0}, v_{2,0}, \ldots, v_{\nu-1,0}$ are periodic.
Let $p_1, p_2, \ldots, p_{\nu -1}$ be their minimal periods respectively. 
By the construction of $P_0$ we have the following critical relations
\begin{itemize}
    \item For $1 \leq i \leq \nu -1$,
    \[ P_0^{p_i + 1}(c_{i,0}) = P_0(c_{i,0}),\] and 
    $P_0^j(c_{i,0}) \notin \Crit(P_0), $ for $1 \leq j \leq p_i-1$.
    \item For $c_{\nu,0}$ we have $P_0^3(c_{\nu,0}) = c_{\nu,0}$, and 
    $P_0(c_{\nu,0}), P_0^2(c_{\nu,0}) \notin \Crit(P_0). $
\end{itemize}
Let $W$ be as before. Each critical relation correspond to subvariety of 
$W$ of codimension one. Let 
\[\cU_{\nu} \= \{ w \in \cU \colon P_w^3(c_{\nu}(P_w)) = c_{\nu}(P_w) \}.\]
Then, $\cU_{\nu}$ is a subvariety of $\cU$ of (real) codimension one.
In the following we write $c_i(w) \= c_i(P_w)$ for every $w \in \cU_{\nu}$.

Consider map $\cR \colon \cU_{\nu} \longrightarrow \R^{\nu}$ defined by 
\[ \cR(w) = (P_w^{p_1+1}(c_1(w)) - P_w(c_1(w)), \ldots, 
P_w^{p_{\nu-1}+1}(c_{\nu -1}(w)) - P_w(c_{\nu-1}(w)), P_w^3(c_{\nu}(w))-
c_{\nu}(w)). \] For this map we 
have that 
\begin{itemize}
    \item $\cR(v_{1,0}, v_{2,0}, \ldots, v_{\nu-1,0}) = (0,0, \ldots, 0),$
    \item $\cR$ is smooth,
    \item $\cR$ has rank $\nu$ at $\vec{v}_0$, in particular is 
    non-degenerated.
\end{itemize}
For the second property see \cite[Appendix A]{Milnor_Tresser_Entropy_monotonicity_for_real_cubic}. For the third property
see \cite[Theorem 6.1]{Levin_Shen_van_Strien_Transversality_elementary_proof}, or
\cite[Theorem C.1]{Levin_Shen_van_Strien_Positive_trans_via_transfer_operator}. By 
considering a smaller $\cU_{\nu}$, we can assume that $\cR$ is a 
diffeomorphism from a neighborhood of $\ov{\cU_{\nu}}$ onto its image.

\begin{proof}[Proof of the Main Theorem]
    
    First, by Proposition \ref{prop:conjugacy_with_Fib_SFT} we have  
    that for every $w \in \cU_{\nu}$, the map $P_w$ has a
    repeller $\Lambda_w$ contained in the laps $I_{\nu-1} \cup I_{\nu}$, and 
    $\Lambda_w$ is homeomorphic to $\Lambda_0$.

    Let $h$ be the holomorphic motion of $\Lambda (P_0)$ given in Proposition
    \ref{prop:conjugacy_with_Fib_SFT}. Fix $\nu -1$ nonperiodic points
    $\xi_1, \xi_2, \ldots, \xi_{\nu-1}$ in $\Lambda(P_0)$, each of them
    close enough to $v_{1,0}, v_{2,0}, \ldots, v_{\nu-1,0}$ 
    respectively, and define the map
    $\widetilde{\cR} \colon \ov{\cU_{\nu}} \longrightarrow \R^{\nu}$, 
    as 
    \[ \widetilde{\cR}(w) = (h(w,\xi_1) - P_w(c_1(P_w)), \ldots, 
    h(w,\xi_{\nu-1}) - P_w(c_{\nu-1}(P_w)), P_w^3(c_{\nu}(w))-c_{\nu}(w)). \]
    Taking $\xi_i$ close enough to $v_{i,0}$ for every $1 \leq i \leq \nu-1$
    and $w$ close enough to $\vec{v}_0$ we can take 
    $\widetilde{\cR}$
    close to $\cR$ in the $C^1$ topology, so that  $\widetilde{\cR}$ 
    is non-degenerated, is a diffeomoprhims from $\cU_{\nu}$ into its
    image, and
    $(0,\ldots,0) \in \widetilde{\cR}_{\lambda}(U_{\nu})$. 
    In particular, there exits a $\hat{w}_0 \in \cU_{\nu}$ so that 
    for every $1 \leq i \leq \nu$ we have $P_{\hat{w}_0}(c_i(\hat{w}_0)) = h(\hat{w}_0, \xi_i)$. Since the itinerary of $h(\hat{w}_0, \xi_i)$, 
    is the same as the itinerary of $\xi_i$ for every
    $1 \leq i \leq \nu -1$, we have that the orbit
    of $c_1(P_{\hat{w}_0}), \ldots, c_{\nu-1}(P_{\hat{w}_0})$ are not
    finite, and they are not in the basin of an attracting periodic point.
    In particular 
    $\cI(c_1(\hat{w}_0)), \ldots, \cI(c_{\nu-1}(\hat{w}_0))$ are not
    eventually periodic.
    Since the points $\xi_1, \xi_2, \ldots, \xi_{\nu-1}$ can be choose in
    in any way (we just need them to be close enough to the points
    $v_{1,0}, v_{2,0}, \ldots, v_{\nu-1,o}$) this prove that the critical
    points undergoes independent bifurcations.

    We can see that the construction carry over if we start at any
    virtually unimodal posct-critically finite polynomial $P_0$ by changing
    the periodic points $v_{1,0}, \ldots v_{\nu,0}$ and also changing 
    the points $x_1, \ldots, x_{\nu}$. Thus, we have infinitely may maps
    in $\cP_d^{\boldsymbol{\mu}}(\R)$ with the same Artin-Mazur zeta 
    function, they are neither post-critically finite, nor hyperbolic, and they are not topologically conjugated.

    Finally, we compute the Artin-Mazur zeta function of $P_w$ for 
    $w \in \cU_{\nu}$ (indeed, the Artin-Mazur zeta function will be the 
    same if we consider $w \in \cU$).
    Since the Artin-Mazur zeta function is invariant under topological 
    conjugacy, by Lemma \ref{prop:conjugacy_with_Fib_SFT}, we have that to 
    compute $\zeta_{F_s}(t)$ we can compute the Artin-Mazur zeta function 
    of the left shift restricted to $\SFib$, and add the contribution of 
    the periodic point in $\cK_{\R}(P_w) \setminus \Lambda_s$. The 
    only periodic points in $\cK_{\R}(P_w) \setminus \Lambda_s$ are the 
    ones in $\partial \cK_{\R}(P_w)$ that form a cycle of period two if 
    $\nu$ is even or it has a single fixed point if $\nu$ is odd, and the
    attracting cycle of period three given by $\cO_{P_w}(c_{\nu}(P_w))$. 

    \begin{equation*}
        \zeta_{P_w}(t) =
        \begin{cases}
            \zeta_{\sigma}(t) \exp \left( \sum_{n \geq 1} \frac{1}{n}t^{2n} \right) \exp \left( \sum_{n \geq 1} \frac{1}{n}t^{3n} \right) & \text{ if } \nu \text{ is even, }\\
            \zeta_{\sigma}(t) \exp \left( \sum_{n \geq 1} \frac{1}{n}t^{n} \right) \exp \left( \sum_{n \geq 1} \frac{1}{n}t^{3n} \right)
            & \text{ if } \nu \text{ is odd.}
        \end{cases}
    \end{equation*}
    Using Lemma \ref{lemm:A-M_zeta_function_of_the_SFT}, we obtain that 
    \[ \zeta_{P_w}(t) = \frac{1}{\Phi_{\nu}(t)(1-t^3)(1-t-t^2)}, \]
    where 
    \begin{equation*}
        \Phi_{\nu}(t) = 
        \begin{cases}
            1-t^2 & \text{ if } \nu \text{ is even, }\\
            1-t & \text{ if } \nu \text{ is odd.}
        \end{cases}
    \end{equation*}
    So, once $\nu \geq$ is fixed, every element in the family $\cU_{\nu}$
    have the same Artin-Mazur zeta function.
    
    This conclude the proof of the Main Theorem.
\end{proof}

\subsection{Proof of Theorem \ref{theo:combinatorial_type_theorem}}
\label{subsect:proof_of_combinatorial_theorem}.
We conclude this section with the proof of Theorem \ref{theo:combinatorial_type_theorem}.
This theorem is consequence of our construction of virtually 
unimodal maps and the unfolding of the critical relations used in the 
proof of the Main Theorem.

\begin{proof}{Proof of Theorem \ref{theo:combinatorial_type_theorem}}
    For part $(i)$ observe that in the construction of the postcritically
    finite virtually unimodal map, we can change the period of the dominant 
    point, and after perturbation we can choose that all the others 
    critical values are in the basin of attraction the dominant critical 
    point. This will give us infinitely many different combinatorial types.

    For part $(ii)$, is enough to choose different periodic points
    $k_{1,0}, k_{2,0}, \ldots, k_{\nu,0}$ in the construction of the 
    virtually unimodal combinatorics. We can also change the period of the
    dominant turning point. Again, this gives rise to infinitely many 
    different combinatorial type.

    For part $(iii)$, is enough to choose different periodic points
    $k_{1,0}, k_{2,0}, \ldots, k_{\nu,0}$ in the construction of the 
    virtually unimodal combinatorics. We can also change the period of the
    dominant turning point. After that, we can perturbe our map as in the 
    proof of the Main Theorem, changing the points 
    $\xi_1,\xi_2, \ldots, \xi_3$. Again, this gives rise to infinitely many 
    different combinatorial type.
\end{proof}

\section{The bimodal case}
\label{section_the_bimodal_case}

We introduce a one-parameter family of cubic polynomials for which every 
element has a rational A-M zeta function. Moreover, infinitely many elements
of this family have a turning point that is not asymptotic to a periodic
cycle; thus, they are not hyperbolic nor PCF.

For $s \in [1 , \infty )$, put
\[ a(s) \= -1 + \frac{1}{s^2(s+1)} \hspace{0.5cm} \text{ and }
\hspace{0.5 cm} b(s) \= -s - \frac{1}{s^2(s+1)}. \]
Consider the real polynomial 
\[F_s(x) = a(s)x^3 + b(s)x^2 +1. \]  

\begin{prop}
\label{prop:cubic_polynomial_family}
There exists $s_* >1$ so that for every $s \in [1,s_*]$ the polynomial $F_s$ 
is cubic with real and distinct critical points, and we have
\[ \zeta_{F_s}(t) = \frac{1}{(1-t^3)(1-t^2)(1-t-t^2)}. \]
Moreover, there are uncountably many $s \in [1,s_*]$ for which $F_s$ is 
neither PCF nor hyperbolic.
\end{prop}

We will split the proof of this proposition into several lemmas. First, we 
will study the structure of the polynomials $F_s$. Then, we will compute
the kneading determinant. Finally, we prove the existence of 
parameters $s$ for which $F_s$ is neither PCF nor hyperbolic.

\begin{lemm}
\label{lemma:geometry_of_cubi_maps_1}
For every $s \in [1,+\infty )$ the following holds:
\begin{enumerate}
    \item $a(s) < 0$ and $b(s) < 0$;
    \item $F_s$ is a cubic polynomial with real critical points $0$ and 
    \[ c_s \= - \frac{2}{3}\frac{b(s)}{a(s)},\] and shape $(-,+,-)$. 
    Furthermore, $c_s \neq 0$, so the critical points of $F_s$ are distinct;
    \item The critical point $0$ is periodic of period three with orbit
    \[ \cO_{F_s}(0) = \{0, F_s(0)=1, F_s^2(0) = -s \}.  \]
\end{enumerate}
Moreover, there is a unique $s_* > 1$ such that $F_{s_*}(c_{s_*}) = 0$, and 
for every $s \in [1,s_*]$ we have:
\begin{enumerate}
    \item[(4)] $c_s \leq -s$ with equality only for $s = s_*$;
    \item[(5)] $s \mapsto F_s(c_s)$ is increasing with 
    $-1 \leq F_s(c_s) \leq 0$, $F_1(c_1) = -1$, and $F_{s_*}(c_{s_*}) = 0$.
    In particular, $F_1$ and $F_{s_*}$ are postcritically finite.
\end{enumerate}

\end{lemm}

\begin{proof}
Let $s \geq 1$. By definition $b(s) < 0$. Also, since $s^2(s+1) >1$, we have 
$a(s) <0$. This proves $(1)$.

For $(2)$, write 
\[ a(s) = - \frac{s^3+s^2-1}{s^2(s+1)}.\] By Descartes' rule of signs, 
$s^3 + s^2 - 1$ has a unique real positive root, then $a(s)$ has a unique 
positive real root. Since $a(1/2) = 5/3 > 0 > -1/2 = a(1)$, the unique 
positive real root of $a(s)$ is in $(0,1)$. So, for $s \geq 1$, the polynomial
$F_s$ is cubic. Taking derivative 
\begin{equation}
    \label{eq:derivative_of_F_s}
    F'_s(x) = 3a(s)x^2 + 2b(s)x = 3a(s)x \left(x + \frac{2}{3} 
    \frac{b(s)}{a(s)}\right).
\end{equation}
From the above, we 
see that the critical points of $F_s$ are $0$, and 
\[c_s = - \frac{2}{3} \frac{b(s)}{a(s)}. \] From $(1)$, we have that 
$c_s <0$. Then, from \eqref{eq:derivative_of_F_s}, we see that $F_s$ is 
decreasing for $x \in (-\infty,c_s) \cup (0, +\infty)$, and is increasing 
for $x \in (c_s,0)$. Thus, the shape of $F_s$ is $(-,+,-)$.

For point $(3)$, we need to compute the image of $0$, and $F_s(0)$. 
It is direct that $F_s(0) = 1$. Also,
\[ F_s(1) = a(s) + b(s) +1 = -s,\] and 
\[F_s(-s) = -a(s)s^3 + b(s)s^2 + 1 =0.\]

Now we prove the existence of $s_*$. Observe that 

\begin{align*}
    F_s(c_s) &= \frac{4b(s)^3}{27a(s)^2} + 1 \\
    &= \frac{-4(s^4+s^3+1)^3 + 27s^2(s+1)(s^3+s^2-1)^2}{27s^2(s+1)(s^3+s^2-1)^2} \\
    &= \frac{-(s^4+s^3-3s-2)^2(4s^4+4s^3-3s+1)}{27s^2(s+1)(s^3+s^2-1)^2}.
\end{align*}
Put \[p(s) = s^4+s^3-3s -2 \hspace{0.5cm} \text{ and } \hspace{0.5cm} 
q(s) = 4s^4 + 4s^3-3s+1.\]
By Descartes' rule of signs, $p(s)$ has a single positive root. Moreover,
as \[p(1) = -2 < 0 < 16 = p(2),\] the unique positive real root of $p(s)$ is
in $(1,2)$. For $q(s)$ we have \[ q'(s) = 16s^3 + 12s^2 -3. \] So $q'(s) > 0$ 
for $s \geq 1$, so $q(s)$ is increasing, and $q(1) = 6$. Them, $q(s)$ has no
real roots in $[1,+\infty )$ (indeed, it has no real roots). This prove that 
there is a unique $s_* >1$ so that $F_{s_*}(c_{s_*}) = 0.$

Let $s \in [1,s_*]$. To prove $(4)$, observe that from $(1)$, it follows that
$c_s < 0$ for $1 \leq s < s_*$.
Then, since $F_s(-s) =0$ and $F_s$ is increasing on $(c_s,0)$, we have that   
$c_s \leq -s$. For $s=s_*$, we get $F_{s_*}(c_{s_*}) = 0 = F_{s_*}(-s_*)$.
Since $F_{s_*}$ is increasing on $(c_{s_*},0)$, we must have 
$c_{s_*} = -s_*$.

Finally, to prove $(5)$, we have 
\[\frac{\partial F_s}{\partial s} (c_s) = \frac{4}{27} 
\frac{3b(s)^2b'(s)a(s)^2 - 2a(s)a'(s)b(s)^3}{a(s)^4},\] with 
\[a'(s) = -\frac{3s+2}{s^3(s+1)^2} \hspace{0.5cm} \text{ and } \hspace{0.5cm} 
b'(s) = -1 + \frac{3s+2}{s^3(s+1)^2}.\] We see that $a'(s) <0$ and 
$b'(s) >0$. This, together with point $(1)$, give us that 
\[\frac{\partial F_s}{\partial s}(c_s) >0.\] So $s \mapsto F_s(c_s)$ is 
increasing.

As $F_1(c_1) = -1 \in \cO_{F_s}(0)$ and $F_{s_*}(c_{s_*}) = 0$, by $(3)$, we 
conclude that $F_1$ and $F_{s_*}$ are PCF.

This finishes the proof of the lemma.



\end{proof}

Numerical computations give us that 
        \[ s_*= 1.371 \ldots\]
From now on, we will consider $s \in [1, s_*]$ unless otherwise stated. 

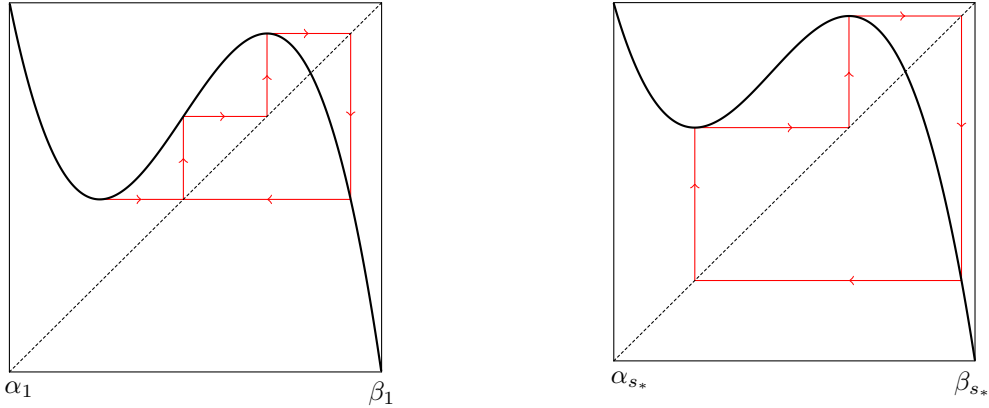
\begin{figure}[h!]
  \raggedright
  \begin{subfigure}[b]{.3 \textwidth}
  \centering
    \begin{tikzpicture}[line cap=round,line join=round,>=triangle 45,x=1.1cm,y=1.1cm]
    \clip(-3.5,-3.5) rectangle (1.5,1.5);
    \draw (-3.08,1.37)-- (1.37,1.37);
      \draw (-3.08,-3.08)-- (1.37,-3.08);
      \draw (-3.08,-3.08)-- (-3.08,1.37);
      \draw (1.37,-3.08)-- (1.37,1.37);
      \draw[-to, color=red] (0,1) -- (0.5,1);
      \draw[ color=red] (0.5,1) -- (1,1);
      \draw[-to, color=red] (1,1) -- (1,0);
      \draw[ color=red] (1,0) -- (1,-1);
      \draw[-to, color=red] (1,-1) -- (0,-1);
      \draw[ color=red] (0,-1) -- (-1,-1);
      \draw[-to, color=red] (-2,-1) -- (-1.5,-1);
      \draw[ color=red] (-1.5,-1) -- (-1,-1);
      \draw[-to, color=red] (-1,-1) -- (-1,-0.5);
      \draw[ color=red] (-1,-0.5) -- (-1,0);
      \draw[-to, color=red] (-1,0) -- (-0.5,0);
      \draw[ color=red] (-0.5,0) -- (0,0);
      \draw[-to, color=red] (0,0) -- (0,0.5);
      \draw[ color=red] (0,0.5) -- (0,1);
      \draw [dash pattern=on 1pt off 1pt] (-3.08,-3.08)-- (1.37,1.37);
      \draw (-3.25,-3.08) node[anchor=north west] {$\alpha_1$};
      \draw (1.1,-3.08) node[anchor=north west] {$\beta_1$};
      \draw[line width=0.8,smooth,samples=100,domain=-3.075:1.365] plot(\x,{-0.5*(\x)^3-1.5*(\x)^2+1});
    \end{tikzpicture}
   \end{subfigure}
   \hspace{3cm}
   \begin{subfigure}[b]{.3 \textwidth}
   \centering
   \begin{tikzpicture}[line cap=round,line join=round,>=triangle 45,x=1.48cm,y=1.48cm]
    \clip(-2.5,-2.5) rectangle (1.25,1.25);
    \draw (-2.09,1.12)-- (1.12,1.12);
      \draw (-2.09,-2.09)-- (1.12,-2.09);
      \draw (-2.09,-2.09)-- (-2.09,1.12);
      \draw (1.12,-2.09)-- (1.12,1.12);
      \draw[-to, color=red] (0,1) -- (0.5,1);
      \draw[ color=red] (0.5,1) -- (1,1);
      \draw[-to, color=red] (1,1) -- (1,0);
      \draw[ color=red] (1,0) -- (1,-1.37);
      \draw[-to, color=red] (1,-1.37) -- (0,-1.37);
      \draw[ color=red] (0,-1.37) -- (-1.37,-1.37);
      \draw[-to, color=red] (-1.37,-1.37) -- (-1.37,-0.5);
      \draw[ color=red] (-1.37,-0.5) -- (-1.37,0);
      \draw[-to, color=red] (-1.37,0) -- (-0.5,0);
      \draw[ color=red] (-0.5,0) -- (0,0);
      \draw[-to, color=red] (0,0) -- (0,0.5);
      \draw[ color=red] (0,0.5) -- (0,1);
      \draw [dash pattern=on 1pt off 1pt] (-2.09,-2.09)-- (1.12,1.12);
      \draw (-2.2,-2.1) node[anchor=north west] {$\alpha_{s_*}$};
      \draw (0.85,-2.1) node[anchor=north west] {$\beta_{s_*}$};
      \draw[line width=0.8,smooth,samples=100,domain=-2.09:1.12] plot(\x,{-0.775615*(\x)^3-1.59538*(\x)^2+1});
    \end{tikzpicture}
   \end{subfigure}
  \label{fig:3}
  \caption{Graph of the map $F_s$ restricted to $\cK_{F_s}$ for $s = 1$
  (left) and $s=s_*$ (right).}
\end{figure}

 \begin{proof}[Proof of Proposition \ref{prop:cubic_polynomial_family}]

The existence of $s_* >1$ and the fact that for every $s \in [1, s_*]$ 
the map $F_s$ is a cubic polynomial with distinct critical points is given
by Lemma \ref{lemma:geometry_of_cubi_maps_1}.

By points $(3)$ and $(5)$ in Lemma \ref{lemma:geometry_of_cubi_maps_1}, we
have that both critical orbits are bounded, so the filled Julia set 
$\cK_{\R}(F_s)$ is connected and since the shape of $F_s$ is $(-,+,-)$, the
boundary forms a cycle of 
period two. Also, by point $(5)$ in Lemma 
\ref{lemma:geometry_of_cubi_maps_1}, $F_s$ is a 
virtually unimodal with dominant turning point $c = 0$ for every 
$s \in [1,s_*]$. Then, the Artin-Mazur zeta function i given by 

\[ \zeta_{F_s}(t) = \frac{1}{(1-t^3)(1-t^2)(1-t-t^2)}. \]

The fact that there are uncountable many $s \in [1,s_*]$ for which 
$F_s$ is neither PCF nor hyperbolic is a consequence of 
Proposition \ref{prop:conjugacy_with_Fib_SFT}. Since for every 
$s \in [1,s_*]$ the map $F_s$ is virtually unimodal with dominan turning 
point periodic of period three, implies that 
$\Lambda (F_s) \cap I_1\cup I_2$ is a cantor set and the action of $F_s$ 
restricted to this set is conjugated to the action of the shift map on 
$\SFib$. Using Proposition \ref{prop:conjugacy_with_Fib_SFT},
we can define a holomorphic motion
$h_s$ of $\Lambda (F_s) \cap [F_s^2(0), F_s(0)] $ over a domain $W_s$ in 
$\C$. Since $[1,s_*]$ is compact
in $\C$ and the collection \[ \{W_s\}_{s \in [1,s_*]} \] is an open 
cover of $[1,s_*]$, so there is a finite collection $s_1,s_2, \ldots s_n$ in 
$[1,s_*]$ such that \[ [1,s_*] \subset \bigcup_{i=1}^n W_{s_i}. \] By
the Identity Principle, if for $i,j \in \{ 1, \ldots,n \}$ we have that 
$W_{s_i} \cap W_{s_j} \neq \emptyset$, then 
$h_i|_{(W_{s_i} \cap W_{s_j})} = h_j|_{(W_{s_i} \cap W_{s_j})}$. 
Fix a $s_0 \in [1,s_*]$, and let $W \= \cup_{i=1}^n W_{s_i}$. The
map $h \colon W \times \Lambda (F_{s_0}) \cap [F_s^2(0), F_s(0)] 
\longrightarrow \CC$, defined by $h(s,x) \= h_{s_i}(s,x)$ if 
$s \in W_{s_i}$.

The map $h$ is a holomorphic motion of $\Lambda(F_{s_0})$ parametrized by 
$W$. For the curve 
\begin{align*}
\Gamma_c \colon [1,s_*] &\longrightarrow \Omega \times \CC \\
s &\mapsto (s,F_s(c_s)),
\end{align*}
we have \[ -1 = \pi_2(\Gamma_c(1)) \leq x \in [F_s^2(0) , 0] \leq 
\pi_2(\Gamma_c(s_*)) = 0. \] Thus, $\Gamma_c$ is transversal to the sections
\[ \{ W \times \{ x \} \colon x \in \Lambda(F_{s_0}) \text{ and } 
(\cI(x))_0 = I_1 \}.\] From the above, we conclude that when restricted to
$[1,s_*]$, for any $\bm{w} \in \SFib$ there is a $s \in [1,s_*]$ such that 
\[ \cI(F_s(c_s)) = \bm{w} \hspace{0.5cm} \text{ or } \hspace{0.5cm} 
\cI(F^2_s(c_s)) = \bm{w}.\]
This conclude the proof of the proposition.

\end{proof}

\bibliography{bibliography}{}
\bibliographystyle{alpha}
\end{document}